 \documentclass[final,3p,times]{elsarticle}

\usepackage{amssymb}
 \usepackage{amsthm}
\usepackage{amsmath,amssymb,amsopn,amsfonts,mathrsfs,amsbsy,amscd}

 \usepackage{lineno}

\usepackage{longtable}
\usepackage{caption}
\usepackage{multirow}

\newcommand{\br}{[\;,\;]}
\newcommand{\too}{\longrightarrow}
\newcommand{\om}{\omega}
\newcommand{\esp}{\quad\mbox{and}\quad}

\newcommand{\G}{{\mathfrak{g}}}

\newcommand{\h}{{\mathfrak{h}}}
\newcommand{\ad}{{\mathrm{ad}}}
\newcommand{\Lu}{{\mathrm{L}}}

\newcommand{\B}{{\cal B}}

\newcommand{\p}{{\mathfrak{p}}}
\newcommand{\ab}{{\mathbf{a}}}
\newcommand{\s}{{\mathbf{s}}}
\newcommand{\rb}{{\mathbf{r}}}

\newcommand{\di}{\displaystyle}

\newcommand{\al}{\alpha}
\newcommand{\be}{\beta}

\newcommand{\ga}{\gamma}

\newcommand{\e}{\epsilon}

\newcommand{\la}{\lambda}
\newcommand{\De}{\Delta}
\newcommand{\de}{\delta}

\font\bb=msbm10

\def\B{\hbox{\bb B}}
\def\R{\hbox{\bb R}}

\newtheorem{Def}{Definition}[section]
\newtheorem{theo}{Theorem}[section]
\newtheorem{pr}{Proposition}[section]

\newtheorem{exem}{Example}

\begin{document}

\begin{frontmatter}


 

\title{ On $k$-para-K\"ahler Lie algebras a subclass of $k$-symplectic Lie algebras}


 \author[label1, label2,label3]{   H. Abchir, Ilham Ait Brik, Mohamed Boucetta}
 \address[label1]{Universit\'e Hassan II\\ Ecole Sup\'erieure de Technologie
 	\\Route d'El Jadida Km 7, B.P. 8012, 20100 Casablanca, Maroc\\
 	{e-mail: h\_abchir@yahoo.com}}
 \address[label2]{Universit\'e Hassan II\\ Facult\'e des Sciences Ain Chock\\e-mail: ilham.aitbrik@gmail.com }

 \address[label3]{Universit\'e Cadi-Ayyad\\
 	Facult\'e des sciences et techniques\\
 	BP 549 Marrakech Maroc\\e-mail: m.boucetta@uca.ac.ma
 }



\begin{abstract} $k$-Para-K\"ahler Lie algebras are a generalization of para-K\"ahler Lie algebras $(k=1)$ and constitute a subclass of $k$-symplectic Lie algebras. In this paper, we show that the characterization of para-K\"ahler Lie algebras as left symmetric bialgebras can be generalized to $k$-para-K\"ahler Lie algebras leading to the introduction of two new structures which are different but both generalize the notion of left symmetric algebra. This permits also the introduction of generalized $S$-matrices.
	We determine then all the $k$-symplectic Lie algebras of dimension $(k+1)$ and all the six dimensional 2-para-K\"ahler Lie algebras.
\end{abstract}


\end{frontmatter}







\section{Introduction}\label{section1}

The $k$-symplectic geometry is a generalization of the symplectic geometry which was developed by  C. G\"unther in \cite{gunther}, A. Awane \cite{awane2} and  Awane and Goze \cite{Goze}  as an attempt to develop a convenient geometric framework to study classical field theories (see \cite{Leon}).
A $k$-symplectic manifold is a smooth manifold $M$ of dimension $(k+1)n$ endowed with an involutive vector subbundle $E\subset TM$ and a family $(\theta^1,\ldots,\theta^k)$ of differential closed  2-forms such that:
 $\mathrm{rank}(E)=nk$, the family $(\theta^1,\ldots,\theta^k)$ is nondegenerate,  i.e.,  $\di\cap_{i=1}^k\ker\theta^i=\{0\}$ and $E$ is isotropic with respect all the $\theta^i$. A left invariant $k$-symplectic structure on a connected   Lie group $G$ of dimension $(k+1)n$ is equivalent to its associated infinitesimal structure, namely, the Lie algebra $\G$ of $G$, a Lie subalgebra $\h$ of dimension $nk$ and a family $\{\theta^i\in\wedge^2\G^*,i=1,\ldots,k\}$ of closed nondegenerate 2-forms  such that $\theta^i_{|\h}=0$ for $i=1,\ldots,n$. We call $(\G,\h,\theta^1,\ldots,\theta^k)$ a $k$-symplectic Lie algebra. If, in addition, there exists a subalgebra $\p\subset\G$ such that $\G=\h\oplus\p$ and $\p$ is isotropic with respect to all the $\theta^i$, we call  $(\G,\h,\p,\theta^1,\ldots,\theta^k)$ a $k$-para-K\"ahler Lie algebra. This terminology is justified by the fact that when $k=1$ we recover the classical notion of para-K\"ahler Lie algebras (see \cite{andrada, bai, bai1,bajo}).
 
 The purpose of this paper is to study $k$-para-K\"ahler Lie algebras aiming the generalization of the results obtained in \cite{bai, bou} in the case of para-K\"ahler Lie algebras. In these papers  para-K\"ahler Lie algebras were considered as left symmetric bialgebras. Roughly speaking, a para-K\"ahler Lie algebra is built from two structures of left symmetric algebras on a vector space and its dual which are compatible in some sense. The compatibility
condition involves representations of Lie algebras and 1-cocycles. This leads naturally to the notion of exact para-K\"ahler Lie algebras (when one 1-cocycle is a coboundary). Exact para-K\"ahler Lie algebras are defined from  $S$-matrices in the same way as exact Lie bialgebras are defined by $R$-matrices.

In this paper, we show that a $k$-para-K\"ahler Lie algebra is build from two new algebraic structures  compatible in some sense (see Theorems \ref{theo1} and \ref{main}). We call them $k$-left symmetric algebra and $(k\times k)$-left symmetric algebra (see Definitions \ref{def1} and \ref{def2}). As for $k=1$ the compatibility condition involves representations and 1-cocycles 	and we have naturally the notion of exact $k$-para-K\"ahler Lie algebras leading to what we call $S_k$-matrices (see Theorem \ref{main1}). The notions of $k$-left symmetric algebra and $(k\times k)$-left symmetric algebra are new and both generalize the notion of left symmetric algebras. We think that these two structures are interesting in their owns. In Proposition \ref{deri}, we give a natural way to build examples of $k$-left symmetric algebras. We give also all $2$-left symmetric algebras in dimension 2 (see Table \ref{1}) and we deduce all the six dimensional $2$-para-K\"ahler Lie algebras (see Table \ref{3}). Our study permits also the determination of all the $k$-symplectic Lie algebras of dimension $(k+1)$ (see Theorems \ref{k=3} and \ref{k}).
	
Section \ref{section2} is devoted to the characterization of $k$-para-K\"ahler Lie algebras by the means of the two new notions of $k$-left symmetric algebra and $(k\times k)$-left symmetric algebras. In Section \ref{section3}, we study exact $k$-para-K\"ahler Lie algebras and we introduce the notion of $S_k$-matrix. Section \ref{section4} is devoted to the determination of the $k$-symplectic Lie algebras of dimension $(k+1)$.
In Section \ref{section5}, we give all the 2-left symmetric algebras of dimension 2 and all the six dimensional $2$-para-K\"ahler Lie algebras.

\paragraph{Convention.} Through this paper, we will deal by many representations of Lie algebras. If $\mu:\G\too\mathrm{End}(V)$ is a representation of a Lie algebra, we denote by $\mu^*:\G\too \mathrm{End}(V^*)$ its dual representation given by
\[ \prec\mu^*(x)(\ga),y\succ=-\prec\ga,\mu(x)(y)\succ,\quad x,y\in V,\ga\in V^*. \]

\section{ Characterization of $k$-para-K\"ahler Lie algebras}\label{section2}

A $k$-symplectic Lie algebra is a real Lie algebra $\G$ of dimension $nk+n$ with a	subalgebra $\h$ of dimension $nk$ and a family $(\theta^1,\ldots,\theta^k)$ of  2-forms satisfying:
\begin{enumerate}\item[$(i)$] The family $(\theta^1,\ldots,\theta^k)$ is nondegenerate, i.e., $\bigcap_{i=1}^k\ker\theta^i=\{0\}$,
	\item[$(ii)$] for $i=1,\ldots,k$, $\theta^i$ is closed, i.e., $\mathrm{d}\theta^i(u,v,w):=\theta^i([u,v],w)+\theta^i([v,w],u)+\theta^i([w,u],v)=0$,
	\item[$(iii)$] $\h$ is totally isotropic with respect to $(\theta^1,\ldots,\theta^k)$, i.e., $\theta^i(u,v)=0$ for any $u,v\in\h$ and for $i=1,\ldots,k$.
	
	\end{enumerate}

According to \cite[Theorem 3.1]{Goze}, there exists a basis $\B^*=(\om^{pi},\om^i)_{1\leq p\leq k,1\leq i\leq n}$ of $\G^*$ such that for any $\al\in\{1,\ldots,k\}$,
\[ \theta^\al=\sum_{i=1}^n\om^{\al i}\wedge\om^i\esp \h=\ker\om^1\cap\ldots\cap\ker\om^n. \]Let $(e_{pi},e_i)_{1\leq p\leq k,1\leq i\leq n}$ be the dual basis of $\B^*$. Then the vector subspace $\mathrm{span}\{e_1,\ldots,e_n \}$ is a supplement of $\h$ and it is totally isotropic with respect to $(\theta^1,\ldots,\theta^k)$. Thus $\h$ has an isotropic supplement. We introduce now the main object of this article.

\begin{Def} Let $(\G,\theta^1,\ldots,\theta^k,\h)$ be a $k$-symplectic Lie algebra. We call it $k$-para-K\"ahler if $\h$ admits an isotropic supplement which is a Lie subalgebra.
	
\end{Def}

When $k=1$, we recover the well-known notion of par-K\"ahler Lie algebras (see \cite{bou}). We proceed now to the study of $k$-para-K\"ahler Lie algebras aiming the generalization of the results obtained for $k=1$ in \cite{bou}.

Let  $(\G,\br,\theta^1,\ldots,\theta^k,\h)$ be a $k$-para-K\"ahler Lie algebra and $\p$  an isotropic Lie subalgebra supplement of $\h$.

The linear map $\Theta:\h\too (\G/\h)^*\times\ldots\times (\G/\h)^*$, $h\mapsto (\Theta_1(h),\ldots,\Theta_k(h))$ where, for any $p\in\G$,
\[ \Theta_\al(h)([p])=\theta^\al(h,p) \]
 is well-defined, injective and for dimensional reasons it is an isomorphism. For any $\al\in\{1,\ldots,k\}$, the vector subspace $\h^\al$  of $\h$ given by
\[ \h^\al=\{h\in\h,\Theta_\be(h)=0,\be=1,\ldots,k,\be\not=\al \}. \]
 has dimension $n$ and $\di\h=\oplus_{\al=1}^k\h^\al$. 
 
 The Lie subalgebra $\h$ carries a product given by
\begin{equation}\label{ls}
\Theta_\al(h_1\bullet h_2)([p])=-\theta^\al(h_2,[h_1,p]),
\end{equation}for any $h_1,h_2\in\h$, for any $p\in\G$ and for any $\al=1,\ldots,k$,

     We have $\G=\h\oplus\p$ and, for any $p\in\p$ and $h\in\h$, the Lie bracket $[p,h]$ can be written
\begin{equation}\label{br} [p,h]=-[h,p]=\phi_\p(h)-\phi_\h(p), \end{equation}where
  $\phi_p(h)\in\h$ and $\phi_h(p)\in\p$.

For any $\al\in\{1,\ldots,k \}$, we define $i_\al:\h^\al\too\p^*$ by putting
\[ i_\al(h)(p)=\theta^\al(h,p). \] 
It is obvious that $i_\al$ is injective and since $\dim\h^\al=\dim\p$ it is bijective. Thus $i_\al^*:\p\too(\h^\al)^*$ is an isomorphism. For any $\be\in\{1,\ldots,k \}$ and for any $p,q\in\p$, the map $h\mapsto-\theta^\be(q,[p,h])$ is an element of $(\h^\al)^*$ and its image by $(i_\al^*)^{-1}$ is an element of $\p$ we denote by $p\star_{\al,\be}q$. Thus,
 for any $\al,\be\in\{1,\ldots,k  \}$, we have a product $\star_{\al,\be}:\p\times\p\too\p$, $(p,q)\mapsto p\star_{\al,\be}q$ where, for any $h\in\h^\al$,
\begin{equation}\label{product}
\theta^\al(p\star_{\al,\be}q,h)=-\theta^\be(q,[p,h]).
\end{equation}
Finally, for any $\al,\be\in \{1,\ldots,k  \}$, we endow $\p^*$ with the product $\bullet_{\al\be}$ by putting, for any $a,b\in\p^*$,
\begin{equation}\label{pr} a\bullet_{\al\be} b=i_\be(i_\al^{-1}(a)\bullet i_\be^{-1}(b) ).  \end{equation}
The formulas \eqref{ls}, \eqref{product} and \eqref{pr} define, respectively, a product on $\h$, a family of products on $\p$ and a family of products on $\p^*$ which depend on $(\theta^1,\ldots,\theta^k)$, the Lie bracket and $\p$. We will use these products to describe the $k$-symplectic Lie algebra in a useful way. Let us give now the properties of these products. Recall that a left symmetric algebra is an algebra $(A,\bullet)$ such that for any $a,b,c\in A$,
\[ \mathrm{ass}(a,b,c)=\mathrm{ass}(b,a,c)\quad\mbox{where}\quad \mathrm{ass}(a,b,c)=(a\bullet b)\bullet c-a\bullet(b\bullet c). \]

\begin{pr}\label{prls}We have:\begin{enumerate}
		\item $(\h,\bullet)$ is a left symmetric algebra, the product $\bullet$ is Lie-admissible, i.e., for any $u,v\in\h$, $[u,v]=u\bullet v-v\bullet u$, and for any $\al=1,\ldots,k$, $\h\bullet\h^\al\subset\h^\al$.
		\item For any $\al,\be\in\{1,\ldots,k \}$ with $\al\not=\be$ and for any $p_1,p_2\in\p$, we have
		\[ [p_1,p_2]=p_1\star_{\al,\al}p_2-p_2\star_{\al,\al}p_1,\; p_1\star_{\al,\be}p_2=p_2\star_{\al,\be}p_1. \]
		\item For any $\al,\be,\ga$, $\bullet_{\al\be}=\bullet_{\al\ga}$ and if we denote $\bullet_{\al\be}=\bullet_{\al}$, we have, for any $a,b,c\in\p^*$,
		\begin{equation}\label{com} a\bullet_{\al}(b\bullet_\be c)-(a\bullet_\al b)\bullet_\be c=
		b\bullet_{\be}(a\bullet_\al c)-(b\bullet_\be a)\bullet_\al c.\end{equation}
	\end{enumerate}	
\end{pr}

\begin{proof}\begin{enumerate}\item We have
	\begin{eqnarray*}
		\Theta_\al(h_1\bullet h_2)([p])-\Theta_\al(h_2\bullet h_1)([p])&=&\theta^\al(h_1,[h_2,p])-\theta^\al(h_2,[h_1,p])\\
		&=&\theta^\al([h_1,h_2],p)=\Theta_\al([h_1,h_2])
	\end{eqnarray*}and hence $\bullet$ is Lie-admissible. On the other hand,
	\begin{eqnarray*}
		\Theta_\al([h_1,h_2]\bullet h_3)&=&-\theta^\al(h_3,[[h_1,h_2],p])\\
		&=&-\theta^\al(h_3,[h_1,[h_2,p]])-\theta^\al(h_3,[h_2,[p,h_1]])\\
		&=&\Theta_\al(h_1\bullet(h_2\bullet h_3))([p])-\Theta_\al(h_2\bullet(h_1\bullet h_3))([p]).
	\end{eqnarray*}So
	\[ [h_1,h_2]\bullet h_3=h_1\bullet(h_2\bullet h_3)-h_2\bullet(h_1\bullet h_3). \]
	This shows that $\bullet$ is left symmetric. It is obvious that $\h\bullet\h^\al\subset\h^\al$.
	\item Now consider $\al,\be\in\{1,\ldots,k \}$, $p,q\in\p$ and $h\in\h^\al$ then
	\begin{eqnarray*}
		0&=&\theta^\be([p,q],h)+\theta^\be([q,h],p)+\theta^\be([h,p],q)\\
		&=&\theta^\be([p,q]_\p,h)+\theta^\al (q\star_{\al,\be}p,h)-\theta^\al (p\star_{\al,\be}q,h).
	\end{eqnarray*}	So if $\al\not=\be$ we get $p\star_{\al,\be}q=q\star_{\al,\be}p$ and if $\al=\be$ we get
	\[ [p,q]_\p=p\star_{\al,\al}q-q\star_{\al,\al}p. \]	
	\item For any $a,b\in\p^*$ and any $q\in\p$,
	\begin{eqnarray*}
		\prec a\bullet_{\al\be} b,q\succ&=&\prec i_\be(i_\al^{-1}(a)\bullet i_\be^{-1}(b) ),q\succ\\
		&=&\theta^\be (i_\al^{-1}(a)\bullet i_\be^{-1}(b),q)\\
		&\stackrel{\eqref{ls}}=&-\theta^\be( i_\be^{-1}(b),[i_\al^{-1}(a),q])\\
		&\stackrel{\eqref{br}}=&\prec b,\phi_q(i_\al^{-1}(a))\succ\\
		&=&\prec a\bullet_{\al\ga} b,q\succ.
	\end{eqnarray*}	We have
	\begin{eqnarray*}
		a\bullet_{\al}(b\bullet_\be c)-(a\bullet_\al b)\bullet_\be c&=&
		a\bullet_{\al\be}(b\bullet_{\be\be} c)-(a\bullet_{\al\be} b)\bullet_{\be\be} c\\
		&=&i_{\be}[i_{\al}^{-1}(a)\bullet(i_\be^{-1}(b)\bullet i_\be^{-1}(c) ) ]
		-i_{\be}[(i_{\al}^{-1}(a)\bullet i_\be^{-1}(b))\bullet i_\be^{-1}(c)  ],\\
		b\bullet_{\be}(a\bullet_\al c)-(b\bullet_\be a)\bullet_\al c&=&b\bullet_{\be\be}(a\bullet_{\al\be} c)-(b\bullet_{\be\al} a)\bullet_{\al\be} c\\&=&
		i_\be[i_\be^{-1}(b)\bullet (i_{\al}^{-1}(a)\bullet i_\be^{-1}(c))]
		-i_\be[(i_\be^{-1}(b)\bullet i_{\al}^{-1}(a))\bullet i_\be^{-1}(c)].
	\end{eqnarray*}This completes the proof.\qedhere

	\end{enumerate}
\end{proof}

We consider  $\Phi(\p,k)=\p\oplus (\p^*)^k$ and we endow $(\p^*)^k$ with the product $\circ$ given by
\begin{equation}\label{circ}
(a_1,\ldots,a_k)\circ (b_1,\ldots,b_k)=\left(\sum_{\al=1}^{k}a_\al\bullet_\al b_1,\ldots,
\sum_{\al=1}^{k}a_\al\bullet_\al b_k\right).
\end{equation}
We define $\phi:(\p^*)^k\otimes \p^*\too \p^*$ and $\psi:\p\otimes \p^k\too \p^k$  by
\begin{equation}\label{action} \begin{cases}\di\phi((a_1,\ldots,a_k),b)=\phi_{(a_1,\ldots,a_k)}b=
\sum_{\al=1}^{k}\Lu_{a_\al}^\al b,\\\di
\psi(q,(p_1,\ldots,p_k))=\psi_q(p_1,\ldots,p_k)=\sum_{\al=1}^{k}\left( \Lu_q^{\al,1}p_\al,\ldots, \Lu_q^{\al,k}p_\al \right).
\end{cases}\end{equation}where $\Lu_{a}^\al:\p^* \too \p^*$, $b\mapsto a\bullet_\al b$ and  
$\Lu_{q}^{\al,\be}:\p \too \p$, $p\mapsto q\star_{\al,\be} p$,	
and we endow $\Phi(\p,k)$ with the bracket $\br_n$
\begin{equation}\label{bracket} \begin{cases}\di
[a,b]_n=a\circ b-b\circ a,\quad\mbox{if}\;a,b\in(\p^*)^k\\
\di[p,q]_n=[p,q],\quad\mbox{if}\;p,q\in\p\\
\di[a,p]_n=\phi_a^*(p)-\psi_p^*a,\quad\mbox{if}\;a\in(\p^*)^k,p\in\p
\end{cases} \end{equation}where
\[ \prec b,\phi_a^*(p)\succ=-\prec\phi_ab,p\succ\esp \prec\psi_p^*a,(p_1,\ldots,p_k)\succ
=-\prec a,\psi_p(p_1,\ldots,p_k)\succ. \]
Finally, we define also a family of 2-forms $\rho^\al$, $\al=1,\ldots,k$ by
\begin{equation}\label{rho}
\rho^\al (p+(a_1,\ldots,a_k),q+(b_1,\ldots,b_k))=\prec a_\al,q\succ-\prec b_\al,p\succ.
\end{equation}
\begin{theo}\label{theo1} $(\Phi(\p,k),\br_n,(\p^*)^k,\rho^1,\ldots,\rho^k)$ is a $k$-para-K\"ahler Lie algebra and $F:\G\too \Phi(\p,k)$, $(h_1+\ldots+h_k+p)\mapsto (p,i_1(h_1),\ldots,i_k(h_k))$ is an isomorphism of $k$-para-K\"ahler Lie algebras. 
	
\end{theo}

\begin{proof} All we need to do is show that $F$ is an isomorphism of Lie algebras which send the $\theta^\al$ to the $\rho^\al$.  Let $h_1\in\h^\al$ and $h_2\in\h^\be$. Then
	\begin{eqnarray*}
		F([h_1,h_2])&=&F (h_1\bullet h_2)-F (h_2\bullet h_1)\\
		&\stackrel{\eqref{pr}}=&F( i_\be^{-1} [i_\al(h_1)\bullet_\al i_\be(h_2)])
		-F ( i_\al^{-1}[i_\be(h_2)\bullet_\be i_\al(h_1)])\\
		&\stackrel{\eqref{circ}}=& F(h_1)\circ F(h_2)-F(h_2)\circ F(h_1).
	\end{eqnarray*}	It is obvious from \eqref{br} and Proposition \ref{prls} that, for any $p,q\in\p$, $F([p,q])=[F(p),F(q)]_n$. Let $h=h_1+\ldots+h_k\in\h$ and $p\in\p$. We have
	\begin{eqnarray*}
		F([h,p])&=&\sum_{\al=1}^kF([h_\al,p])\\
		&=& \sum_{\al=1}^k F([h_\al,p]_{\p})- \sum_{\al=1}^k F([p,h_\al]_{\h})\\
		&\stackrel{\eqref{br}}=&\sum_{\al=1}^k\phi_{h_\al}(p)-\sum_{\al=1}^k
		\left(i_1(\phi_{p}(h_\al)),\ldots,i_k(\phi_{p}(h_\al))\right).
	\end{eqnarray*}
	
	For any $a\in\p^*$ and $\be\in\{1,\ldots,k \}$, we have
	\begin{eqnarray*}
		\prec a,\phi_{h_\al}(p)\succ&=&\theta^\be (i_\be^{-1}(a),\phi_{h_\al}(p))\\
		&=&\theta^\be (i_\be^{-1}(a),[h_\al,p])\\
		&\stackrel{\eqref{ls}}=&-\theta^\be (h_\al\bullet i_\be^{-1}(a),p )\\
		&\stackrel{\eqref{pr}}=&-\prec i_{\al}(h_\al)\bullet_\al a,p\succ \\
		&=&\prec a,(\mathrm{L}_{F(h_\al)^\al}^\al)^*p\succ.
	\end{eqnarray*}
	Thus $\phi_{h_\al}(p)=(\mathrm{L}_{F(h_\al)^\al}^\al)^*p$. On the other hand, for any $q\in\p$, \begin{eqnarray*}
		\prec i_\be(\phi_{p}(h_\al)),q\succ&=&\theta^\be( [h_\al,p],q)\\
		&\stackrel{\eqref{product}}=&\theta^\al (p\star_{\al,\be}q,h_\al)\\
		&=&-\prec i_\al(h_\al),p\star_{\al,\be}q\succ\\
		&=&\prec (\mathrm{L}_{p}^{\al,\be})^*(i_\al(h_\al)) ,q \succ.	
	\end{eqnarray*}
		Thus $i_\be(\phi_{p}(h_\al))=(\mathrm{L}_{p}^{\al,\be})^*(i_\al(h_\al))$.
Therefore 
	\begin{eqnarray*}
		F([h,p])&=&\sum_{\al=1}^k\phi_{h_\al}(p)-\sum_{\al=1}^k
		\left(i_1(\phi_{p}(h_\al)),\ldots,i_k(\phi_{p}(h_\al))\right).\\	
		&=&\sum_{\al=1}^k (\mathrm{L}_{F(h_\al)^\al}^\al)^*p-\sum_{\al=1}^k
		\left((\mathrm{L}_{p}^{\al,1})^*(i_\al(h_\al)),\ldots,(\mathrm{L}_{p}^{\al,k})^*(i_\al(h_\al))\right)\\
		&=& \phi_{F(h)}^*(F(p))-\psi_{F(p)}^*(F(h))
	\end{eqnarray*}
		 and we get that $F([h,p])=[F(h),F(p)]_n$. To conclude one can check easily that $F$ send the $\theta^\al$ to the $\rho^\al$.
\end{proof}

To study the converse, we introduce two algebraic structures which appeared naturally in our study above.
 
 \begin{Def}\label{def1} A $k$-left symmetric algebra is a real vector space $\mathcal{A}$ endowed with $k$ left symmetric products $\bullet_1,\ldots,\bullet_k$ such that one of the following equivalent assertions hold:
 	\begin{enumerate}\item For any $\al,\be\in\{1,\ldots,k \}$ and for any $a,b,c\in\mathcal{A}$,
 		\begin{equation}\label{com1} a\bullet_{\al}(b\bullet_\be c)-(a\bullet_\al b)\bullet_\be c=
 		b\bullet_{\be}(a\bullet_\al c)-(b\bullet_\be a)\bullet_\al c.\end{equation}
 		\item $(\mathcal{A}^k,\circ)$ is a left symmetric algebra where $\circ$ is given by
 		\begin{equation}\label{circ1}
 		(a_1,\ldots,a_k)\circ (b_1,\ldots,b_k)=\left(\sum_{\al=1}^{k}a_\al\bullet_\al b_1,\ldots,
 		\sum_{\al=1}^{k}a_\al\bullet_\al b_k\right).
 		\end{equation}
 	\end{enumerate}In this case the map $\phi:\mathcal{A}^k\times \mathcal{A}\too \mathcal{A}$ given by
 		\begin{equation}\label{phi}
 		 \phi((a_1,\ldots,a_k),b)=\phi_{(a_1,\ldots,a_k)}b=\sum_{\al=1}^{k}\Lu_{a_\al}^\al b 
 		\end{equation}defines a representation of the Lie algebra $(\mathcal{A}^k,\br)$ in $\mathcal{A}$ where $[a,b]=a\circ b-b\circ a$.

 	\end{Def}

Indeed,  the two relations \eqref{com1} and \eqref{circ1} are equivalent. In fact, for any $ a,b,c\in \mathcal{A}^{k} $ we have
 
 \begin{align*}
 (a\circ b)\circ c &= \left( \sum\limits_{\alpha=1}^{k}(a_{\alpha}\bullet_{\al} b_{1}),...,   \sum\limits_{\alpha=1}^{k}(a_{\alpha}\bullet_{\al} b_{k}) \right)\circ c\\
 &=\big(  \sum\limits_{\alpha=1}^{k} \sum\limits_{\beta=1}^{k}(a_{\alpha}\bullet_{\al} b_{\beta})\bullet_{\be} c_{1} , .... ,\sum\limits_{\alpha=1}^{k} \sum\limits_{\beta=1}^{k}(a_{\alpha}\bullet_{\al} b_{\beta})\bullet_{\be} c_{k}  \big),\\
  a\circ (b\circ c)&= a\circ \left( \sum\limits_{\beta=1}^{k}(b_{\beta} \bullet_{\be} c_{1}),...,\sum\limits_{\beta=1}^{k}(b_{\beta} \bullet_{\be} c_{k})   \right)\\
 &= \left(  \sum\limits_{\alpha=1}^{k} \sum\limits_{\beta=1}^{k} a_{\alpha}\bullet_{\al}(b_{\beta}\bullet_{\be}c_{1}),...,  \sum\limits_{\alpha=1}^{k} \sum\limits_{\beta=1}^{k} a_{\alpha}\bullet_{\al}(b_{\beta}\bullet_{\be}c_{k}) \right).
 \end{align*}
 Then
 $$\mathrm{ass}(a,b,c)=  \left( \sum\limits_{\alpha=1}^{k} \sum\limits_{\beta=1}^{k}[(a_{\alpha}\bullet_{\al} b_{\beta})\bullet_{\be} c_{1}-a_{\alpha}\bullet_{\al}(b_{\beta}\bullet_{\be}c_{1}) ],...,\sum\limits_{\alpha=1}^{k} \sum\limits_{\beta=1}^{k}[(a_{\alpha}\bullet_{\al} b_{\beta})\bullet_{\be} c_{k}-a_{\alpha}\bullet_{\al}(b_{\beta}\bullet_{\be}c_{k}) ]  \right),   $$
  where $\mathrm{ass}(a,b,c)=(a\circ b)\circ c-a\circ (b\circ c)$.
 Similarly, we have 
 $$\mathrm{ass}(b,a,c)=  \left( \sum\limits_{\alpha=1}^{k} \sum\limits_{\beta=1}^{k}[(b_{\beta}\bullet_{\be} a_{\alpha})\bullet_{\al} c_{1}-b_{\beta}\bullet_{\be}(a_{\alpha}\bullet_{\al}c_{1}) ],...,\sum\limits_{\alpha=1}^{k} \sum\limits_{\beta=1}^{k}[(b_{\beta}\bullet_{\be} a_{\alpha})\bullet_{\al} c_{k}-b_{\beta}\bullet_{\be}(a_{\alpha}\bullet_{\al}c_{k}) ]  \right).   $$
 But $\circ$ is left symmetric if and only if, for any $\al,\be\in\{1,\ldots,k\}$ and any $a=(0,\ldots,a_\al,\ldots,0)$, $b=(0,\ldots,b_\al,\ldots,0)$ and $c\in\mathcal{A}^k$, $\mathrm{ass}(a,b,c)=\mathrm{ass}(b,a,c)$ which gives the equivalence.
 Moreover, if \eqref{com1} or \eqref{circ1} holds then, for any $a,b\in \mathcal{A}^{k}$ and any $c\in \mathcal{A}$, we have
 \begin{align*}
 \phi([a,b],c) &= \phi(a\circ b,c)-\phi(b\circ a,c)\\
 &=\sum_{\al=1}^{k}\left( \phi((a_\al\bullet_\al b_1,\ldots,a_\al\bullet_\al b_k),c)-\phi((b_\al\bullet_\al a_1,\ldots,b_\al\bullet_\al a_k),c)\right)\\
 &=\sum_{\al=1}^{k}\sum_{\be=1}^{k}\left(L^{\be}_{a_\al\bullet_\al b_{\be}}c -L^{\be}_{b_\al\bullet_\al a_{\be}}c   \right)\\
 &\stackrel{\eqref{com1}}=\sum_{\al=1}^{k}\sum_{\be=1}^{k}\left( L^{\al}_{a_{\al}}(L^{\be}_{b_{\be}}c)-L^{\be}_{b_{\be}}(L^{\al}_{a_{\al}}c) \right)\\
 &= \left[\sum_{\al=1}^{k} L^{\al}_{a_{\al}},\sum_{\be=1}^{k}L^{\be}_{b_{\be}}\right]c\\
 &=\left[\phi(a,.),\phi(b,.)\right](c).
 \end{align*}
 This shows that $\phi$ defines a representation of the Lie algebra $(\mathcal{A}^k,\br)$ in $\mathcal{A}$.

	\begin{Def}\label{def2} A $(k\times k)$-left symmetric algebra is a  vector space $\mathcal{B}$ endowed   with a $k\times k$-matrix  $(\star_{\al,\be})_{1\leq\al,\be\leq k}$ of products  such that:
		\begin{enumerate}\item For any $\al,\be$ and for any $p,q\in\mathcal{B}$, 
			$$p
			\star_{\al,\al}q-q\star_{\al,\al}p=p
			\star_{\be,\be}q-q\star_{\be,\be}p=[p,q].$$
		\item $\star_{\al,\be}$ are commutative when $\al\not=\be$,
		\item  the map $\psi:\mathcal{B}\otimes \mathcal{B}^k\too \mathcal{B}^k$ given by
	\begin{equation}\label{psi} \psi(q,(p_1,\ldots,p_k))=\psi_q(p_1,\ldots,p_k)=\left(\begin{array}{ccc}
	\Lu_q^{1,1}&\ldots&\Lu_q^{k,1}\\\vdots&&\vdots\\\Lu_q^{1,k}&\ldots&\Lu_q^{k,k}
	\end{array}   \right)\left(\begin{array}{c}p_1\\\vdots\\p_k \end{array} \right)=
	\sum_{\al=1}^{k}\left( \Lu_q^{\al,1}p_\al,\ldots, \Lu_q^{\al,k}p_\al \right) \end{equation}satisfies
	\begin{equation}\label{ps} \psi_{[p,q]}=[\psi_p,\psi_q]. \end{equation}

	\end{enumerate}
	\end{Def}

	\begin{pr}Let $\mathcal{B}$ be $(k\times k)$-left symmetric algebra. Then, the relation \eqref{ps} is equivalent to
		\[  \Lu_{[u,v]}^{\al,\ga}=\sum_{\be=1}^{k}\left[\Lu_u^{\be,\ga}\circ \Lu_v^{\al,\be}
		-\Lu_v^{\be,\ga}\circ \Lu_u^{\al,\be}\right], \]for any $u,v\in\p$ and any $\al,\ga$. Moreover, $\star_{\al,\al}$ is a Lie admissible and  $\br$ is a Lie bracket and hence $\psi$ is a representation of a Lie algebra.
		
	\end{pr}
	\begin{proof} We have the curvature of $\star_{\al,\al}$ is given by
		\[ R^{\al}(u,v)w=\Lu_{[u,v]}^{\al,\al}w-[\Lu_u^{\al,\al},\Lu_v^{\al,\al}](w)=
		\sum_{\al\not=\be=1}^{k}\left[\Lu_u^{\be,\al}\circ \Lu_v^{\al,\be}w
		-\Lu_v^{\be,\al}\circ \Lu_u^{\al,\be}w\right]. \]One can deduce easily that
		\[ \oint_{u,v,w}R^{\al}(u,v)w=\oint_{u,v,w}[[u,v],w]=0. \]

	\end{proof}
	
	\begin{exem}\begin{enumerate}
		\item Note that the notions of $1$-left symmetric algebra and $(1\times 1)$-left symmetric algebra are the same and correspond to the classical notion of left symmetric algebra.
		 \item If $(\mathcal{A},\bullet)$ is a left-symmetric algebra, then $(\mathcal{A},\bullet_1=\bullet,\ldots,\bullet_k=\bullet)$ is a $k$-left symmetric algebra.
		 \item If $\bullet_1,\ldots,\bullet_k$ are left symmetric products and $\mathcal{B}$ such that $a\bullet_\al b-b\bullet_\al a=a\bullet_\be b-b\bullet_\be a$ for any $\al,\be$ then $(\mathcal{B},(\star_{\al,\be})_{1\leq\al\leq\be\leq k})$ is $(k\times k)$-left symmetric algebra where $\star_{\al,\be}=0$ if $\al\not=\be$ and 
		 $\star_{\al,\al}=\bullet_\al$.
		 		 \end{enumerate}

	\end{exem}

It turns out that as a para-K\"ahler Lie algebra is built from two compatible left symmetric algebras (see \cite{bai, bou}), a $k$-para-K\"ahler Lie algebras is built from a $k$-left symmetric algebra and a $(k\times k)$-left symmetric algebra compatible in some sense.

Let $\p$ be a vector space of dimension $n$ such that:
\begin{enumerate}\item  $\p$ carries a structure  $(\br_\p,(\star_{\al,\be})_{1\leq\al,\be\leq k})$ of $(k\times k)$-left symmetric algebra,\item
 $\p^*$ carries a structure $(\bullet_1,\ldots,\bullet_k)$ of $k$-left symmetric algebra.\end{enumerate}
 Define on $\Phi(\p,k)=\p\oplus (\p^*)^k$ the bracket
\begin{equation}\label{bracket1} \di
[a,b]=a\circ b-b\circ a,
\di[p,q]=[p,q]_\p\esp 
\di[a,p]=\phi_a^*(p)-\psi_p^*a,\quad a,b\in (\p^*)^k, p,q\in\p,
 \end{equation} and the family $(\rho^1,\ldots,\rho^k)$  of 2-forms given by \eqref{rho}.
 
   The vector space $(\p^*)^k$ has a structure of Lie algebra $\br$ and $\phi$ is a representation of this Lie algebra and $\p$ has a structure of Lie algebra $\br_\p$ and $\psi$ is a representation of this Lie algebra structure. We denote by  $\phi^T:\p\too \p^{k}\otimes\p$ and $\psi^T:(\p^*)^k\too (\p^*)\otimes(\p^*)^k$ the dual of $\phi:(\p^*)^k\otimes\p^*\too \p^*$ and $\psi:\p\otimes\p^k\too\p$.

 The following theorem is a generalization of a result first obtained in \cite[Theorem 4.1]{bai} and recovered in \cite[Proposition 3.3]{bou}. The proof is similar. 
 \begin{theo}\label{main} $(\Phi(\p,k),\br)$ is a Lie algebra if and only if
 	\begin{enumerate}\item $\phi^T:\p\too \p^{k}\otimes\p$ is a 1-cocycle of $(\p,\br_\p)$ and the representation $\psi\otimes\ad$, i.e.,
 		\begin{eqnarray*}
 			{\phi}^T([p,q]_\p)((a_1,\ldots,a_k),b)&=&{\phi}^T(p)((a_1,\ldots,a_k),
 			\ad_q^*b)+{\phi}^T(p)(\psi_q^*(a_1,\ldots,a_k),b)
 			-{\phi}^T(q)((a_1,\ldots,a_k),\ad_p^*b)\\&&-{\phi}^T(q)(\psi_p^*(a_1,\ldots,a_k),b).
 			\end{eqnarray*}
 		\item $\psi^T:(\p^*)^k\too (\p^*)\otimes(\p^*)^k$ is a 1-cocycle of $((\p^*)^k,\br)$ and the representation $\phi\otimes\ad$ and $\br$ is given by
 		$[a,b]=a\circ b-b\circ a$, i.e.,
 		\begin{eqnarray*}
 			{\psi}^T([a,b])(p,(q_1,\ldots,q_k))&=&{\psi}^T(a)(p,\ad_b^*(q_1,\ldots,q_k))+{\psi}^T(a)(\phi_b^*p,(q_1,\ldots,q_k))
 			-{\psi}^T(b)(p,\ad_a^*(q_1,\ldots,q_k))\\&&-{\psi}^T(b)(\phi_a^*p,(q_1,\ldots,q_k)).
 				\end{eqnarray*}
 			\end{enumerate}In this case $(\Phi(\p,k),\br,(\p^*)^k,\rho^1,\ldots,\rho^k)$ is a $k$-para-K\"ahler Lie algebra. Moreover, all $k$-para-K\"ahler Lie algebras are obtained in this way.
 	
 	\end{theo}

 	A $(k\times k)$-left symmetric algebra structure on $\p$ and a $k$-left symmetric left algebra structure on $\p^*$ are called compatible if they satisfy the conditions of Theorem \ref{main}.

 	\begin{exem}\begin{enumerate}\item Any $k$-left symmetric  algebra structure on $\p^*$ is compatible with the trivial $(k\times k)$-left symmetric algebra structure on $\p$.

\item Any $(k\times k)$-left symmetric algebra structure on $\p$ is compatible with the trivial $k$-left symmetric algebra structure on $\p^*$.
\end{enumerate}
\end{exem} 	
 	
 	We end this section by giving a way of building $k$-symplectic left algebras. The following is a generalization of a construction given by S. Gelfand.

 			\begin{pr}\label{deri} Let $(A,.)$ be a commutative associative algebra and $D_1,D_2$ two derivations of $(A,.)$ which commute. Then the two products
 				\[ a\bullet_1 b=a.D_1b\esp a\bullet_2 b=a.D_2b  \]
 				are left symmetric and compatible.
 			\end{pr}
 			
 			\begin{proof} We have, for any $a,b\in A$, the left multiplication of $\bullet_i$ is
 				$\mathrm{L}_a^i=\mathrm{L}_a\circ D_i$, where $\mathrm{L}_a$ is the left multiplication of the product on $A$. Then
 				\begin{eqnarray*}
 					Q&=&[\mathrm{L}_a\circ D_1,\mathrm{L}_a\circ D_2]-\mathrm{L}_{a.D_1b}\circ D_2+\mathrm{L}_{b.D_2a}\circ D_1\\
 					&=&	\mathrm{L}_a\circ D_1\circ\mathrm{L}_a\circ D_2-\mathrm{L}_a\circ D_2\circ\mathrm{L}_a\circ D_1-\mathrm{L}_{a}\circ \mathrm{L}_{D_1b}\circ D_2+\mathrm{L}_{b}\circ\mathrm{L}_{D_2a} \circ D_1\\
 					&=&\mathrm{L}_a\circ D_1\circ\mathrm{L}_a\circ D_2-\mathrm{L}_a\circ D_2\circ\mathrm{L}_a\circ D_1-\mathrm{L}_{a}\circ [D_1,\mathrm{L}_{b}]\circ D_2+\mathrm{L}_{b}\circ[D_2,\mathrm{L}_{a}] \circ D_1\\
 					&=&\mathrm{L}_{a}\circ \mathrm{L}_{b}\circ D_1\circ D_2-\mathrm{L}_{b}\circ \mathrm{L}_{a}\circ D_2\circ D_1\\
 					&=&\mathrm{L}_{ab}\circ[D_1,D_2]=0.
 				\end{eqnarray*}	
 				
 			\end{proof}
 		
 		\begin{exem}\label{exem2} We consider $\R^4$ endowed with the associative commutative product
 			\[ e_1.e_1=e_1,\; e_1.e_2=e_2.e_1=e_2,\; e_1.e_3=e_3.e_1=e_3,\; e_1.e_4=e_4.e_1=e_4. \]We consider the two derivations
 			\[ D_1=\left[ \begin {array}{cccc} 0&0&0&0\\ \noalign{\medskip}0&1&0&0
 			\\ \noalign{\medskip}0&0&1&0\\ \noalign{\medskip}0&0&0&1\end {array}
 			\right] \esp D_2=\left[ \begin {array}{cccc} 0&0&0&0\\ \noalign{\medskip}0&0&a&b
 			\\ \noalign{\medskip}0&0&0&c\\ \noalign{\medskip}0&0&0&0\end {array}
 			\right]. 
 			 \]These two derivations commute and, according to Proposition \ref{deri}, they define a $2$-left symplectic structure on $\R^4$ by
 			 \[ e_1\bullet_1e_i=e_i,i=2,3,4,\quad e_1\bullet_2 e_3=ae_2\esp e_1\bullet_2 e_4=be_2+ce_3. \]

 			\end{exem}	
  
\section{Exact $k$-para-K\"ahler Lie algebras }\label{section3}

In this section, we start with a $k$-left symmetric algebra structure on $\p^*$ and we look for a compatible $(k\times k)$-left symmetric algebra structure on $\p$ such that $\psi^T$ is a coboundary leading to the generalization of the results obtained in the case $k=1$ in \cite{bai, bou}.

Let $\p$ be a vector space of dimension $n$. Suppose that $\p^*$ is endowed with a $k$-left symmetric algebra structure $(\bullet_1,\ldots,\bullet_k)$  and we consider $\phi:(\p^*)^k\times \p^*\too\p^*$ the associated representation given by
\[ \phi(a,\rho)=\Lu_a\rho=\sum_{\al=1}^{k}\mathrm{L}^\al_{a_\al}\rho=\sum_{\al=1}^{k} a_\al\bullet_\al \rho,\;\quad a=(a_1,\ldots,a_k)\in(\p^*)^k,\rho\in\p^*.\]
The left symmetric product on $(\p^*)^k$ is given by
\[ a\circ b=\Lu_ab=(\phi(a,b_1),\ldots,\phi(a,b_k))\esp [a,b]=a\circ b-b\circ a. \]

Let $\rb\in\p^*\otimes (\p^*)^k$  and we define $\psi:\p\otimes(\p)^k\too\p^k$ by
\[ \prec a,\psi(p,u)\succ=-\rb(\phi_a^*p,u)-\rb(p,\ad_a^*u),\; \quad p\in\p,u\in\p^k,a\in(\p^*)^k. \]
If we define  $\rb_\#:\p\too (\p^*)^k$ by $\prec \rb_\#(p),u\succ =\rb(p,u)$, we get
\begin{equation} \label{eqpsi}\prec a,\psi(p,u)\succ=\Lu_a(\rb)(p,u)+\prec a,\Lu_{\rb_\#(p)}^*u\succ,\;  \end{equation}where
\[ \Lu_a(\rb)(p,u)=-\rb(\Lu_a^*p,u)-\rb(p,\Lu_a^*u). \]Note that
\[ \rb(p,u)=\sum_{\al=1}^{k}r(p,(0,\ldots,u_\al,\ldots,0))=\sum_{\al=1}^{k}\rb_\al(p,u_\al)=\sum_{\al=1}^{k}
\left(\ab_\al(p,u_\al)+\s_\al(p,u_\al)\right),\quad  \]
where $\rb_\al=\ab_\al+\s_\al\in\p^*\otimes\p^*$,  $\ab_\al$ is its skew-symmetric part and $\s_\al$ is its symmetric part. On the other hand, we define the family of products $\star_{\al,\be}$ on $\p$ by
\[ \psi(p,u)=\sum_{\al=1}^{k}\psi(p,(0,\ldots,u_\al,\ldots,0))=\sum_{\al=1}^{k}(p\star_{\al,1}u_\al,\ldots,p\star_{\al,k}u_\al). \]

Let us  see now under which conditions $\psi$ defines a $(k\times k)$-left symmetric algebra structure on $\p$ such that $\phi^T$ is a 1-cocycle of $(\p,\br_\p)$ and the representation $\psi\otimes\ad$.

Let us start by studying under which conditions the family of products $\star_{\al,\be}$ is adapted to our purpose.

\begin{pr} For any $\al,\be\in\{1,\ldots,k \}$ with $\al\not=\be$ and for any $\rho\in\p^*$ and $p,q\in\p$,
\[ \prec \rho,p\star_{\al,\al}q\succ=\Lu_{\rho}^\al(\rb_\al)(p,q)+
\prec \rho,\Lu_{\rb_\#(p)}^*q\succ\esp \prec \rho,p\star_{\al,\be}q\succ=\Lu_{\rho}^\be(\rb_\al)(p,q).  \]
Thus $\star_{\al,\be}$ is commutative when $\al\not=\be$ if and only if $\Lu_{\rho}^\be(\ab_\al)=0$ and
\[ \prec \rho,p\star_{\al,\be}q\succ=\Lu_{\rho}^\be(\s_\al)(p,q). \]
Moreover, $[p,q]_\p=p\star_{\al,\al}q-q\star_{\al,\al}p$ is independent of $\al$ if and only if $\Lu_{\rho}^\al(\ab_\al)=\Lu_{\rho}^\be(\ab_\be)$ for any $\al,\be$. In this case
\[ \prec \rho,[p,q]_\p\succ=2\Lu_{\rho}^\al(\ab_\al)(p,q)+\prec \rho,\Lu_{\rb_\#(p)}^*q-\Lu_{\rb_\#(q)}^*p\succ. \]
	
	\end{pr}

	\begin{proof} For any $\al\in\{1,\ldots,k \}$ and $\rho\in\p^*$, we denote $\rho^\al\in(\p^*)^k$ with $\rho^\al_i=\rho\de_{i\al}$. 
		
		Put
		\[ \psi(p,h)=\sum_{\al=1}^{k}\psi(p,(0,\ldots,h_\al,\ldots,0))=\sum_{\al=1}^{k}(p\star_{\al,1}h_\al,\ldots,p\star_{\al,k}h_\al). \]
		
		We have, for any $\rho\in\p^*$ and $p,q\in\p$ and $\al,\be\in\{1,\ldots,k \}$
		\[ \prec \rho,p\star_{\al,\be}q\succ= \prec \rho^\be,\psi(p,q^\al)\succ=
		\Lu_{\rho^\be}(\rb)(p,q^\al)+\prec \rho^\be,\Lu_{\rb_\#(p)}^*q^\al\succ.
		\]But $\prec \rho^\be,\Lu_{\rb_\#(p)}^*q^\al\succ=\de_{\al\be}\prec \rho,\Lu_{\rb_\#(p)}^*q\succ$ and
		\[ \Lu_{\rho^\be}(\rb)(p,q^\al)=-\rb((\Lu_\rho^\be)^*p,q^\al)-\rb(p,\Lu_{\rho^\be}^*q^\al)
		=-\rb_\al((\Lu_\rho^\be)^*p,q)-\rb_\al(p,(\Lu_{\rho}^\be)^*q). \]
		Thus
		\[ \prec \rho,p\star_{\al,\al}q\succ=\Lu_{\rho}^\al(\rb_\al)(p,q)+
			\prec \rho,\Lu_{\rb_\#(p)}^*q\succ,\;  \]
		and if $\al\not=\be$,
		\[ \prec \rho,p\star_{\al,\be}q\succ=\Lu_{\rho}^\be(\rb_\al)(p,q). \]So we get
		\[ \prec \rho,[p,q]_\p\succ=2\Lu_{\rho}^\al(\ab_\al)(p,q)+\prec \rho,\Lu_{\rb_\#(p)}^*q-\Lu_{\rb_\#(q)}^*p\succ,\rho\in\p^*,p,q\in\p. \]
		So we must have $\Lu_{\rho}^\al(\ab_\al)=\Lu_{\rho}^\be(\ab_\be)$ for any $\al,\be$.
		Thus $\star_{\al,\be}$ is commutative when $\al\not=\be$ if and only if $\Lu_{\rho}^\be(\ab_\al)=0$ and
		\[ \prec \rho,p\star_{\al,\be}q\succ=\Lu_{\rho}^\be(\s_\al)(p,q). \]\end{proof}
		
		We suppose that, for any $\al,\be\in\{1,\ldots,k \}$ with $\al\not=\be$ and for any $\rho\in(\p)^*$,
		\[ \Lu_{\rho}^\be(\ab_\al)=0\esp \Lu_{\rho}^\al(\ab_\al)=\Lu_{\rho}^\be(\ab_\be). \]
		So we get
		\[ \Lu_a(\ab)(p,u)=\sum_{\al,\be}\Lu_{a_\al}^\al(\ab_\be)(p,u_\be)=
			\sum_{\al=1}^{k}
			\Lu_{a_\al}^\al(\ab_\al)(p,u_\al). \]
		So
		\begin{equation}\label{a}\prec \rho,[p,q]_\p\succ=2\Lu(\ab)(\rho,p,q)+\prec \rho,\Lu_{\rb_\#(p)}^*q-\Lu_{\rb_\#(q)}^*p\succ,\rho\in\p^*,p,q\in\p.\end{equation}where $\Lu(\ab)\in \p\otimes \p^*\otimes\p^*$ is given by
		\[ \Lu(\ab)(\rho,p,q)=\Lu_{\rho}^\al(\ab_\al)(p,q),\quad \al=1,\ldots,k. \]

		The second step is see under which conditions $\psi$ defines a $(k\times k)$-left symmetric algebra structure on $\p$ and $\phi^T$ is a 1-cocycle of $(\p,\br_\p)$ and the representation $\psi\otimes\ad$, i.e.,

		\[\begin{cases} \prec a,\psi([p,q]_\p,u)\succ=\prec a,\psi(p,\psi(q,u))\succ
		-\prec a,\psi(q,\psi(p,u))\succ,\\
		Q:={\phi}^T([p,q]_\p)(a,\rho)-{\phi}^T(p)(a,
		\ad_q^*\rho)-{\phi}^T(p)(\psi_q^*a,\rho)
		+{\phi}^T(q)(a,\ad_p^*\rho)+{\phi}^T(q)(\psi_p^*a,\rho)=0\end{cases}
		 \]for any $a\in(\p^*)^k$, $u\in\p^k$ and $p,q\in\p$, $\rho\in\p^*$ and  
		\[ \psi_p^*a=-(\Lu_a(\rb))_\#(p)+\Lu_{\rb_\#(p)}a=-(\Lu_a(\rb))_\#(p)+\rb_\#(p)\circ a\esp \phi^T(p)(a,\rho)=\prec\Lu_a\rho,p\succ. \]

Now
\begin{eqnarray*}
	Q&=&{\phi}^T([p,q]_\p)(a,\rho)-{\phi}^T(p)(a,
	\ad_q^*\rho)-{\phi}^T(p)(\psi_q^*a,\rho)
	+{\phi}^T(q)(a,\ad_p^*\rho)+{\phi}^T(q)(\psi_p^*a,\rho)\\
	&=&\prec \Lu_a\rho,[p,q]_\p\succ-\prec \Lu_a\ad_q^*\rho,p\succ+\prec \Lu_{(\Lu_a(\rb))_\#(q)}\rho,p\succ-\prec\Lu_{\rb_\#(q)\circ a}\rho,p\succ\\
	&&+\prec \Lu_a\ad_p^*\rho,q\succ-\prec \Lu_{(\Lu_a(\rb))_\#(p)}\rho,q\succ+\prec\Lu_{\rb_\#(p)\circ a}\rho,q\succ.
\end{eqnarray*} Then
\begin{eqnarray*}
Q&=&2\Lu(\ab)(\Lu_a\rho,p,q)-\prec \Lu_a\rho,\Lu_{\rb_\#(q)}^*p\succ
+\prec \Lu_a\rho,\Lu_{\rb_\#(p)}^*q\succ-\prec \rho,[q,\Lu_a^*p]_\p\succ\\
&&+\prec \Lu_{(\Lu_a(\rb))_\#(q)}\rho,p\succ-\prec\Lu_{\rb_\#(q)\circ a}\rho,p\succ-\prec\Lu_{(\Lu_a(\rb))_\#(p)}\rho,q\succ+\prec\Lu_{\rb_\#(p)\circ a}\rho,q\succ\\
&&+\prec \rho,[p,\Lu_a^*q]_\p\succ\\
&=&2\Lu(\ab)(\Lu_a\rho,p,q)+2\Lu(\ab)(\rho,p,\Lu_a^*q)+2\Lu(\ab)(\rho,\Lu_a^*p,q)\\&&+\prec \rho,\Lu_{\rb_\#(p)}^*\Lu_a^*q\succ-\prec \rho,\Lu_{\rb_\#(\Lu_a^*q)}^*p\succ
-\prec \rho,\Lu_{\rb_\#(q)}^*\Lu_a^*p\succ+\prec \rho,\Lu_{\rb_\#(\Lu_a^*p)}^*q\succ\\
&&+\prec \rho,\Lu_a^*\Lu_{\rb_\#(q)}^*p\succ
-\prec \rho,\Lu_a^*\Lu_{\rb_\#(p)}^*q\succ\\
&&-\prec \rho,\Lu_{(\Lu_a(\rb))_\#(q)}^*p\succ+\prec \rho,\Lu_{\rb_\#(q)\circ a}^*p\succ+\prec \rho,\Lu_{(\Lu_a(\rb))_\#(p)}^*q\succ-\prec \rho,\Lu_{\rb_\#(p)\circ a}^*q\succ\\
&=&2\Lu(\ab)(\Lu_a\rho,p,q)+2\Lu(\ab)(\rho,p,\Lu_a^*q)+2\Lu(\ab)(\rho,\Lu_a^*p,q)+\prec \rho,A(p,q)-A(q,p)\succ,
\end{eqnarray*}where
\[ A(p,q)=\Lu_{\rb_\#(p)}^*\Lu_a^*q+\Lu_{\rb_\#(\Lu_a^*p)}^*q-\Lu_a^*\Lu_{\rb_\#(p)}^*q+
\Lu_{(\Lu_a(\rb))_\#(p)}^*q-\Lu_{\rb_\#(p)\circ a}^*q. \]

But
\begin{equation}\label{lu} \rb_\#(\Lu_{a}^* p)+(\Lu_a(\rb))_\#(p)= a\circ \rb_\#(p) \end{equation}and hence
\[ A(p,q)=[\Lu_{\rb_\#(p)}^*,\Lu_a^* ]q-\Lu_{[\rb_\#(p),a]}^*q=0. \]
\[Q= 2\Lu(\ab)(\Lu_a\rho,p,q)+2\Lu(\ab)(\rho,p,\Lu_a^*q)+2\Lu(\ab)(\rho,\Lu_a^*p,q). \]
Now let us compute
\[ A= \prec a,\psi([p,q]_\p,u)\succ-\prec a,\psi(p,\psi(q,u))\succ
+\prec a,\psi(q,\psi(p,u))\succ. \] We have, for $p,q\in\p$,  $u\in\p^k$ and $a\in(\p^*)^k$.

\begin{eqnarray*}
A_1&=&\prec a,\psi([p,q]_\p,u)\succ\\
&=&\Lu_{a}(\rb)([p,q]_\p,u)+\prec a,\Lu_{\rb_\#([p,q]_\p)}^*u\succ\\
&=&-\rb(\Lu_a^*[p,q]_\p,u)-\rb([p,q]_\p,\Lu_a^*u)+\prec a,\Lu_{\rb_\#([p,q]_\p)}^*u\succ,\\
A_2&=&-\prec a,\psi(p,\psi(q,u))\succ\\
&=&-\prec (\Lu_a(\rb))_\#(p),\psi(q,u)\succ+\prec \rb_\#(p)\circ a,\psi(q,u)\succ\\
&=&-\Lu_{(\Lu_a(\rb))_\#(p)}(\rb)(q,u)-\prec \Lu_a(\rb)_\#(p),\Lu_{\rb_\#(q)}^*u\succ+
\Lu_{\rb_\#(p)\circ a}(\rb)(q,u)+\prec \rb_\#(p)\circ a,\Lu_{\rb_\#(q)}^*u\succ\\
&\stackrel{\eqref{lu}}=&-\Lu_{a\circ \rb_\#(p)}(\rb)(q,u)+\Lu_{\rb_\#(\Lu_a^*p)}(\rb)(q,u)-\Lu_a(\rb)(p,\Lu_{\rb_\#(q)}^*u)-
\rb(\Lu_{\rb_\#(p)\circ a}^*q,u)-\rb(q,\Lu_{\rb_\#(p)\circ a}^*u)\\
&&-\prec  a,\Lu_{\rb_\#(p)}^*\circ \Lu_{\rb_\#(q)}^*u\succ.
\end{eqnarray*}So
\begin{eqnarray*}
A_1&=&-\rb(\Lu_a^*[p,q]_\p,u)-\rb([p,q]_\p,\Lu_a^*u)+\prec a,\Lu_{\rb_\#([p,q]_\p)}^*u\succ,\\
A_2&=&\rb(\Lu_{a\circ \rb_\#(p)}^*q,u)+\rb(q,\Lu_{a\circ \rb_\#(p)}^*u)-\rb(\Lu_{\rb_\#(\Lu_a^*p)}^*q,u)-\rb(q,\Lu_{\rb_\#(\Lu_a^*p)}^*u)\\
&&+\rb(\Lu_a^*p,\Lu_{\rb_\#(q)}^*u)+\rb(p,\Lu_a^*\Lu_{\rb_\#(q)}^*u)-
\rb(\Lu_{\rb_\#(p)\circ a}^*q,u)-\rb(q,\Lu_{\rb_\#(p)\circ a}^*u)
-\prec  a,\Lu_{\rb_\#(p)}^*\circ \Lu_{\rb_\#(q)}^*u\succ\\
&=&\rb(\Lu_{[a, \rb_\#(p)]}^*q,u)+\rb(q,\Lu_{[a,\rb_\#(p)]}^*u)-\rb(\Lu_{\rb_\#(\Lu_a^*p)}^*q,u)
-\rb(q,\Lu_{\rb_\#(\Lu_a^*p)}^*u)\\
&&+\rb(\Lu_a^*p,\Lu_{\rb_\#(q)}^*u)+\rb(p,\Lu_a^*\Lu_{\rb_\#(q)}^*u)
-\prec  a,\Lu_{\rb_\#(p)}^*\circ \Lu_{\rb_\#(q)}^*u\succ\\
A_3&=&\prec a,\psi(q,\psi(p,u))\succ\\&=&-\rb(\Lu_{[a, \rb_\#(q)]}^*p,u)-\rb(p,\Lu_{[a, \rb_\#(q)]}^*u)+\rb(\Lu_{\rb_\#(\Lu_a^*q)}^*p,u)+\rb(p,\Lu_{\rb_\#(\Lu_a^*q)}^*u)\\
&&-\rb(\Lu_a^*q,\Lu_{\rb_\#(p)}^*u)-\rb(q,\Lu_a^*\Lu_{\rb_\#(p)}^*u)
+\prec  a,\Lu_{\rb_\#(q)}^*\circ \Lu_{\rb_\#(p)}^*u\succ
\end{eqnarray*} Thus if we put $\Delta(\rb)(p,q)=\rb_\#([p,q]_\p)-[\rb_\#(p),\rb_\#(q)]$ then
\begin{eqnarray*}
A&=&-\rb(\Lu_a^*[p,q]_\p,u)-\rb([p,q]_\p,\Lu_a^*u)+\prec a,\Lu_{\Delta(\rb)(p,q)}^*u\succ\\
&&+\rb(\Lu_{[a, \rb_\#(p)]}^*q,u)-\rb(q,\Lu_{\rb_\#(p)}^*\Lu_a^*u)-\rb(\Lu_{\rb_\#(\Lu_a^*p)}^*q,u)
-\rb(q,\Lu_{\rb_\#(\Lu_a^*p)}^*u)\\
&&+\rb(\Lu_a^*p,\Lu_{\rb_\#(q)}^*u)
-\rb(\Lu_{[a, \rb_\#(q)]}^*p,u)+\rb(p,\Lu_{ \rb_\#(q)}^*\circ \Lu_{a}^*u)+\rb(\Lu_{\rb_\#(\Lu_a^*q)}^*p,u)+\rb(p,\Lu_{\rb_\#(\Lu_a^*q)}^*u)\\
&&-\rb(\Lu_a^*q,\Lu_{\rb_\#(p)}^*u)\\
&=&\rb(s,u)	-\rb([p,q]_\p,\Lu_a^*u)+\prec a,\Lu_{\Delta(\rb)(p,q)}^*u\succ
+\rb(p,\Lu_{ \rb_\#(q)}^*\circ \Lu_{a}^*u)-\rb(q,\Lu_{\rb_\#(p)}^*\Lu_a^*u)\\
&&+\rb(\Lu_a^*p,\Lu_{\rb_\#(q)}^*u)-\rb(q,\Lu_{\rb_\#(\Lu_a^*p)}^*u)
+\rb(p,\Lu_{\rb_\#(\Lu_a^*q)}^*u)
-\rb(\Lu_a^*q,\Lu_{\rb_\#(p)}^*u),
\end{eqnarray*}	with 
$$s=-\Lu_a^*[p,q]_\p+\Lu_{[a, \rb_\#(p)]}^*q-\Lu_{[a, \rb_\#(q)]}^*p-\Lu_{\rb_\#(\Lu_a^*p)}^*q+
\Lu_{\rb_\#(\Lu_a^*q)}^*p.$$
We have
\begin{eqnarray*}
q_1:&=&	-\rb([p,q]_\p,\Lu_a^*u)
+\rb(p,\Lu_{ \rb_\#(q)}^*\circ \Lu_{a}^*u)-\rb(q,\Lu_{\rb_\#(p)}^*\Lu_a^*u)+\prec a,\Lu_{\Delta(\rb)(p,q)}^*u\succ\\
&=&-\prec \rb_\#([p,q]_\p),\Lu_a^*u\succ-\prec [\rb_\#(q),\rb_\#(p)],\Lu_a^*u\succ+\prec a,\Lu_{\Delta(\rb)(p,q)}^*u\succ\\
&=&\prec [a,\De(\rb)(p,q)],u\succ.
		\end{eqnarray*}
On the other hand
\begin{eqnarray*}
	d=\rb(p,\Lu_{\rb_\#(\Lu_a^*q)}^*u)-\rb(\Lu_a^*q,\Lu_{\rb_\#(p)}^*u)	
	=\prec \rb_\#(p),\Lu_{\rb_\#(\Lu_a^*q)}^*u\succ-\prec\rb_\#(\Lu_a^*q),\Lu_{\rb_\#(p)}^*u\succ
	=\prec[\rb_\#(p),\rb_\#(\Lu_a^*q)],u\succ.
\end{eqnarray*}
So
\[ A=\rb(s,u)+\prec [a,\De(\rb)(p,q)],u\succ+\prec[\rb_\#(p),\rb_\#(\Lu_a^*q)],u\succ
-\prec[\rb_\#(q),\rb_\#(\Lu_a^*p)],u\succ. \]

  From \eqref{a}, we have
\begin{equation}\prec \rho,[p,q]_\p\succ=2\Lu(\ab)(\rho,p,q)+\prec \rho,\Lu_{\rb_\#(p)}^*q-\Lu_{\rb_\#(q)}^*p\succ,\rho\in\p^*,p,q\in\p.\end{equation}where $\Lu(\ab)\in \p\otimes \p^*\otimes\p^*$ is given by
\[ \Lu(\ab)(\rho,p,q)=\Lu_{\rho}^\al(\ab_\al)(p,q),\quad \al=1,\ldots,k. \]

\begin{eqnarray*}
-\Lu_a^*[p,q]_\p&=&2\Lu(\ab)(\Lu_a\bullet,p,q)-\Lu_a^*\Lu_{\rb_\#(p)}^*q
+\Lu_a^*\Lu_{\rb_\#(q)}^*p,\\
-\Lu_{\rb_\#(\Lu_a^*p)}^*q&=&[q,\Lu_a^*p]_\p-\Lu_{\rb_\#(q)}^*\Lu_a^*p
+2\Lu(\ab)(\bullet,\Lu_a^*p,q),\\
\Lu_{\rb_\#(\Lu_a^*q)}^*p&=&-[p,\Lu_a^*q]+\Lu_{\rb_\#(p)}^*\Lu_a^*q+
2\Lu(\ab)(\bullet,p,\Lu_a^*q).
\end{eqnarray*}

We have
\begin{eqnarray*}
s&=&
2\Lu(\ab)(\Lu_a\bullet, p,q)+2\Lu(\ab)(\bullet,\Lu_a^*p,q)+2\Lu(\ab)(\bullet,p,\Lu_a^*q)-[\Lu_a^*p,q]_\p-[p,\Lu_a^*q]_\p.
\end{eqnarray*}	So, we get
\[ A=\prec[a,\Delta(r)(p,q)]+\rb_\#(P(a)(\Lu(\ab)) )-\De(\rb)(\Lu_a^*p,q)-
\De(\rb)(p,\Lu_a^*q),u\succ=0, \]where
\[\prec\rho, P(a)(\Lu(\ab))\succ=2\Lu(\ab)(\Lu_a\rho, p,q)+2\Lu(\ab)(\rho,\Lu_a^*p,q)+2\Lu(\ab)(\rho,p,\Lu_a^*q). \]

So far, we have proved the following theorem.
\begin{theo}\label{main1}Let $\p$ be a vector space of dimension $n$ such that   $\p^*$ is endowed with a $k$-left symmetric algebra structure $(\bullet_1,\ldots,\bullet_k)$ and $\rb=(\s_1+\ab_1,\ldots,\s_k+\ab_k)\in\p^*\otimes(p^*)^k$ such that, for any $\al\not=\be$ and for any $\rho\in\p^*$,
	\[ \Lu_{\rho}^\al(\ab_\be)=0\esp \Lu_{\rho}^\al(\ab_\al)=\Lu_{\rho}^\be(\ab_\be)=:\Lu(\ab)(\rho,.,.). \] 
Then $\psi$ given by \eqref{eqpsi}	defines a $(k\times k)$-left symmetric structure on $\p$ compatible with the $k$-left symmetric structure of $(\p^*)^k$ if and only if, for any $a\in(p^*)^k$ and $p,q\in\p$,
\[ [a,\Delta(\rb)(p,q)]+\Lu_a(\Delta(\rb))(p,q)=0,\quad \Delta(\rb)(p,q)=\rb_\#([p,q]_\p)-[\rb_\#(p),\rb_\#(q)] \]and,  for any $a\in(\p^*)^k$, $\rho\in\p^*$, $p,q\in\p$,  
	\[ \Lu(\ab)(\Lu_a\rho,p,q)+\Lu(\ab)(\rho,\Lu_a^*p,q)+\Lu(\ab)(\rho,p,\Lu_a^*q)=0.
	 \]
\end{theo}

An  important consequence of this theorem is the introduction of the generalization of $S$-matrices (see \cite{bou, bai}).

\begin{Def}
	Let $\rb=(\rb^1,\ldots,\rb^k)$ be a family of symmetric elements of $\mathcal{A}\otimes\mathcal{A}$ where $\mathcal{A}$ has a structure of $k$-left symmetric algebra $(\bullet_1,\ldots,\bullet_k)$. We call $\rb$ a  $S_k$-matrix if, for any $\al=1,\ldots,k$,  $p,q\in \mathcal{A}^*$,
	\[ \rb^\al_\#([p,q]_*)=\sum_{\be=1}^{k}
	\left[\rb^\be_\#(p)\bullet_\be \rb^\al_\#(q)-\rb^\be_\#(q)\bullet_\be \rb^\al_\#(p)\right], \]where
	\[ [p,q]_*=\sum_{\be=1}^{k}\left[(\Lu_{\rb^\be_\#(p)}^\be)^*q-(\Lu_{\rb^\be_\#(q)}^\be)^*p\right]. \]\end{Def}

\begin{exem}\begin{enumerate}\item Let $(\mathcal{A},\bullet)$ be a left symmetric algebra and $\rb\in \mathcal{A}\otimes \mathcal{A}$ be a classical $S$-matrix, i.e.,
	$\rb$ satisfies
	\[ \rb\left(\mathrm{L}_{\rb_\#(p)}^*q-\mathrm{L}_{\rb_\#(p)}^*q\right)=\rb_\#(p)\bullet \rb_\#(q)-\rb_\#(q)\bullet \rb_\#(p), \]for any $p,q\in	\mathcal{A}^*$ (see \cite{bai, bou}).
		 For any $k\geq1$, endow $\mathcal{A}$ with the $k$-left symmetric structure given by $\bullet_\al=\mu_\al \bullet$, where $\mu_\al\in\R$. Then $\rb^k=(\rb,\ldots,\rb)$ is a  $S_k$-matrix.
	\item Consider the 2-left symmetric on $\R^4$ given in Example \ref{exem2}, then one can check by a direct computation that
	\[ \rb^1=r_{2,4}e_2\odot e_4+r_{2,2}e_2\odot e_2+r_{4,4}e_4\odot e_4\esp \rb^2=s_{1,1}e_1\otimes e_1+s_{1,2}e_1\odot e_2 \]
	constitute a $S_2$-matrix on $\R^4$ ($\odot$ is the symmetric product).
	\end{enumerate}
	
	\end{exem}

Let $(\p^*,\bullet_1,\ldots,\bullet_k)$ be a $k$-left symmetric algebra	and  $\rb=(\rb^1,\ldots,\rb^k)\in\p^*\otimes(p^*)^k$.
	We call $\rb$ a quasi-$S_k$-matrix if, for any $\al,\be$ and for any $\rho\in\p^*$, $a\in(\p^*)^k$, $p,q\in\p$,
\[\Lu_{\rho}^\al(\ab_\be)=0,\; [a,\Delta(\rb)(p,q)]+\Lu_a(\Delta(\rb))(p,q)=0,\quad \Delta(\rb)(p,q)=\rb_\#([p,q]_\p)-[\rb_\#(p),\rb_\#(q)]. \]
According to Theorem \ref{main}, $(\Phi(\p,k)=(\p^*)^k\oplus\p,(\p^*)^k,\br^{\rb},\theta^1,\ldots,\theta^k)$ is a $k$-para-K\"ahler Lie algebra where
\[ [a+p,b+q]^{\rb}=\left\{[a,b]+\psi_p^*b-\psi_q^*a\right\}+\left\{\phi_a^*q-\phi_b^*p+[p,q]_\p\right\},\;a,b\in(\p^*)^k,p,q\in\p \]and
\[ \begin{cases}\di [a,b]=a\circ b-b\circ a,\;a\circ b=\sum_{\al=1}^k(a_\al\bullet_\al b_1,\ldots,a_\al\bullet_\al b_k),\\\di
\prec\rho,\phi_a^*p\succ=-\sum_{\al=1}^k\prec a_\al\bullet_\al \rho,p\succ,\\\di
\psi_p^*a=\rb_\#(\phi_a^*p)+[\rb_\#(p),a],\\\di
[p,q]_\p=\phi_{\rb_\#(p)}^*q-\phi_{\rb_\#(q)}^*p=\sum_{\be=1}^{k}\left[(\Lu_{\rb_\be(p)}^\be)^*q-(\Lu_{\rb_\be(q)}^\be)^*p\right],\\\di
\theta^i(a+p,b+q)=\prec a_i,q\succ-\prec b_i,p\succ.
\end{cases} \]

Note also that $\Phi(\p,k)$ has another Lie algebra structure, namely,
\[ [a+p,b+q]^\rhd=[a,b]+\phi_a^*q-\phi_b^*p. \]
We consider now the bracket on $\Phi(\p,k)$ given by
\[ [a+p,b+q]^{\rhd,\rb}=[a+p,b+q]^\rhd+\De(\rb)(p,q). \]
\begin{theo} The linear map $K:(\Phi(\p,k),\br^{\rhd,\rb})\too(\Phi(\p,k),\br^\rb)$, $a+p\mapsto a-\rb_\#(p)+p$ is an isomorphism of Lie algebras.
	
	\end{theo}
	\begin{proof} It is clear that $K$ is bijective and that for any $a,b\in(\p^*)^k$, $K([a,b]^{\rhd,\rb})=[K(a),K(b)]^\rb$. Now for any $p,q\in\p$,
		\begin{eqnarray*}
			K([p,q]^{\rhd,\rb})&=&\De(\rb)(p,q),\\
			\;[K(p),K(q)]^\rb&=&[p-\rb_\#(p),q-\rb_\#(q)]^\rb\\
			&=&[p,q]_\p-\psi_p^*\rb_\#(q)+\psi_q^*\rb_\#(p)-\phi_{\rb_\#(p)}^*q+
			\phi_{\rb_\#(q)}^*p+[\rb_\#(p),\rb_\#(q)]\\
			&=&-\rb_\#\left(\phi_{\rb_\#(q)}^*p \right)-[\rb_\#(p),\rb_\#(q)]
			+\rb_\#\left(\phi_{\rb_\#(p)}^*q \right)+[\rb_\#(q),\rb_\#(p)]+[\rb_\#(p),\rb_\#(q)]\\
			&=&\De(\rb)(p,q).
				\end{eqnarray*}On the other hand,
			\begin{eqnarray*}
				K([a,p]^{\rhd,\rb})&=&K(\phi_a^*p)=\phi_a^*p-\rb_\#(\phi_a^*p)\\
				\;[K(a),K(p)]^\rb&=&[a,p-\rb_\#(p)]^\rb\\
				&=&-[a,\rb_\#(p)]+\phi_a^*p-\psi_p^*a\\
				&=&-[a,\rb_\#(p)]+\phi_a^*p-\rb_\#(\phi_a^*p)-[\rb_\#(p),a]\\
				&=&\phi_a^*p-\rb_\#(\phi_a^*p).
				\end{eqnarray*}	This completes the proof.
		\end{proof}
		
		\begin{pr} Let $(\p^*,\bullet_1,\ldots,\bullet_k)$ be a $k$-left symmetric Lie algebra and $\rb=(\rb_1,\ldots,\rb_k)$ be a quasi-$S_k$-matrix. Then $(\Phi(\p,k),\br^{\rhd,\rb},(\p^*)^k,\theta^1_\rb,\ldots,\theta^k_\rb)$ is a $k$-para-K\"ahler Lie algebra where
			\[ \theta^i_\rb(a+p,b+q)=\theta^i(a+p,b+q)-2\s_i(p,q), \]where
			$\s_i$ is the symmetric part of $\rb_i$.
			
			\end{pr}

			\section{$k$-symplectic Lie algebras of dimension $(k+1)$}\label{section4}
			In \cite{Awane}, there is a study of $k$-symplectic Lie algebras of dimension $(k+1)$.  In this section, by using Theorem \ref{main}, we give a   description of these Lie algebras which completes the results obtained in \cite{Awane}.

			Let $(\G,\h,\theta^1,\ldots,\theta^k)$ be a $k$-symplectic Lie algebra of dimension $(k+1)$. Since $\h$ has codimension 1 then $\G$ is indeed a $k$-para-K\"ahler Lie algebra  and according to Theorem \ref{main}, there exists a basis $(f_1,\ldots,f_k,e)$ of $\G$ and $(a_1,\ldots,a_k)\in\R^k$ such that $(f_1,\ldots,f_k)$ is a basis of $\h$ and  for any $i,j=1,\ldots,k$
			\begin{equation} \label{eq}[f_i,f_j]=a_if_j-a_jf_i, [e,f_i]=a_i e+D(f_i),\;\theta^i=f_i^*\wedge e^* \end{equation}where $D:\h\too\h$ is an homomorphism. This bracket must satisfy the Jacobi identity. We will solve the obtained equations in what follows.  We distinguish two cases:  $k=2$ and $k\geq3$. Note first that if we define $\ell\in\h^*$ by $\ell(f_i)=a_i$ the bracket above satisfies
			\begin{equation}\label{bra} [x,y]=\ell(x)y-\ell(y)x\esp [e,x]=\ell(x)e+D(x) \end{equation}for any $x,y\in\h$ and one can see easily that the Jacobi identity is equivalent to
			\begin{equation}\label{jacobi}
			\ell(y)D(x)-\ell(x)D(y)+\ell(D(y))x-\ell(D(x))y=0
			\end{equation}for any $x,y\in\h$.

			Let us start with with the case $k=2$. We consider $\mathrm{sl}(2,\R)$ with its   basis
			$\left\{h=\left(\begin{matrix}1&0\\0&-1\end{matrix} \right), g=\left(\begin{matrix}0&1\\0&0\end{matrix} \right), f=\left(\begin{matrix}0&0\\1&0\end{matrix} \right)\right\},$ where
			\begin{equation*} \label{eqsl2} [h,g]=2g,\;[h,f]=-2f\esp [g,f]=h. \end{equation*}

			 We consider the Lie algebra  $\mathfrak{sol}=\left\{\left(\begin{matrix} x&0&y\\0&-x&z\\0&0&0\end{matrix}\right),x,y,z\in\R   \right\}$. In the basis
			\[ \left\{u_1=\left(\begin{matrix} 1&0&0\\0&-1&0\\0&0&0\end{matrix}\right),
			u_2=\left(\begin{matrix} 0&0&1\\0&0&0\\0&0&0\end{matrix}\right),\;
			u_3=\left(\begin{matrix} 0&0&0\\0&0&1\\0&0&0\end{matrix}\right)\right\}, \]we have
			\begin{equation*} \label{eqsol} [u_1,u_2]=u_2,\; [u_1,u_3]=-u_3\esp[u_2,u_3]=0. \end{equation*}
			
			\begin{theo}\label{k=3} Let $(\G,\h,\theta^1,\theta^2)$ be a $2$-symplectic Lie algebra of dimension 3. Then   one of the following situations holds:
			\begin{enumerate}	\item $\h$ is an abelian ideal and there exists a basis $(e,f,g)$ of $\G$ and $D$ an endomorphism of $\h$ such that $[h,e]=D(h)$ for any $h\in\h$,  $\theta^1=e^*\wedge f^*$ and $\theta^2=e^*\wedge g^*$.
				\item $(\G,\h,\theta^1,\theta^2)$ is isomorphic to $(\mathrm{sl}(2,\R),\h_0,\rho^1,\rho^2)$ with $\h_0=\mathrm{span}\{h,g\}$, $\rho^1=h^*\wedge f^*+b g^*\wedge f^*$ and $\rho^2=g^*\wedge f^*$. 
				\item $(\G,\h,\theta^1,\theta^2)$ is isomorphic to $(\mathfrak{sol},\h_0,\rho^1,\rho^2)$ with $\h_0=\mathrm{span}\{u_1,u_2\}$, $\rho_1=u_1^*\wedge u_3^*+b u_2^*\wedge u_3^*$ and $\theta^2=cu_1^*\wedge u_3^*+u_2^*\wedge u_3^*$.
				
				\end{enumerate}
				\end{theo}
				
				\begin{proof} According to what above, $\G=\h\oplus\R e$ and there exists a basis $(f_1,f_2)$ of $\h$, a homomorphism $D:\h\too\h$ and $\ell\in\h^*$ such that 
					\[ \ell(f_1)=a_1,\ell(f_2)=a_2,\; \theta^1=f_1^*\wedge e^*\esp \theta^2=f_2^*\wedge e^* \]and the Lie bracket is given by \eqref{bra} and $\la$ and $D$ satisfy \eqref{jacobi}. 
					
					If $\ell=0$ then this equation is satisfied and $\G$ is a central extension of an abelian Lie algebra. 
					
					If $\ell\not=0$ we can suppose that $a_1\not=0$.
					Put $g_2=a_2f_1-a_1f_2\in\ker\ell$ and the equation \eqref{jacobi} is equivalent to
					\[ a_1D(g_2)+\ell(D(f_1))g_2-\ell(D(g_2))f_1=0. \]
					Put $D(f_1)=d_{11}f_1+d_{21}g_2$ and $D(g_2)=d_{12}f_1+d_{22}g_2$ then the equation above is equivalent to $d_{11}=-d_{22}$.
					So in the basis $(f_1,g_2,e)$, we have
					\[ [e,f_1]=a_1e+d_{11}f_1+d_{21}g_2,[e,g_2]=d_{12}f_1-d_{11}g_2\esp [f_1,g_2]=a_1g_2 \]and
					\[ \theta^1=f_1^*\wedge e^*+a_2g_2^*\wedge e^*\esp \theta^2=-a_1g_2^*\wedge e^*. \]
					We distinguish two cases:
					
					$\bullet$ $d_{12}\not=0$. If we put
					\[ (h,g,f)=\left(\frac2{a_1}f_1,g_2,-\frac{1}{a_1d_{12}}\left( 2e+\frac{2d_{11}}{a_1}f_1+\frac{d_{12}}{a_1}g_2 \right)
					      \right) \]we get the desired isomorphism between $\G$ and $\mathrm{sl}(2,\R)$.
					      
					$\bullet$ $d_{12}=0$.  If we put
					\[ (u_1,u_2,u_3)=\left(-\frac1{d_{11}}e,g_2,a_1e+d_{11}f_1+\frac{d_{21}}2g_2      \right) \]we get the desired isomorphism between $\G$ and $\mathfrak{sol}$.    
					\end{proof}

				\begin{theo}\label{k} Let $(\G,\h,\theta^1,\ldots,\theta^k)$ be a $k$-symplectic Lie algebra such that $\dim\h=k\geq3$. Then one of the following situation holds:
					\begin{enumerate}\item $\h$ is an abelian ideal and there exists a basis $(e,f_1,\ldots,f_k)$ of $\G$ and an endomorphism $D$ of $\h$ such that $\h=\mathrm{span}\{f_1,\ldots,f_k\}$,  $[e,h]=D(h)$ for any $h\in\h$ and,
						for $\al=1,\ldots,k$, $\theta^\al=f_\al^*\wedge e^*$.
						
						\item	
						There exists  $(f_1,\ldots,f_k,e)$ a basis of $\G$, a family of constants $(a_1,\ldots,a_k)\in\R^k$, $a_1\not=0$,  $(b_2,\ldots,b_k)\in\R^{k-1}$ and $\la\in\R$ such that $\h=\mathrm{span}\{f_1,\ldots,f_k\}$, 
						\[ \theta^1=f_1^*\wedge e^*-\sum_{i=2}^ka_if_i^*\wedge e^*\esp \theta^i=a_1f_i^*\wedge e^*, i=2,\ldots,k, \]
						and the non vanishing Lie brackets are given by
						\[ [e,f_1]=a_1e+\la f_1+\sum_{l=2}^kb_lf_l,\; [e,f_i]=-\la f_i,\;[f_1,f_i]=a_1f_i, i=2,\ldots,k. \]
					\end{enumerate}	
				\end{theo}
				
				\begin{proof}  According to what above, $\G=\h\oplus\R e$ and there exists a basis $(f_1,\ldots,f_k)$ of $\h$, a homomorphism $D:\h\too\h$ and $\ell\in\h^*$ such that 
					\[ \ell(f_i)=a_i\esp \theta^i=f_i^*\wedge e^*,\quad i=1,\ldots,k \]and the Lie bracket is given by \eqref{bra} where $\ell$ and $D$ satisfy \eqref{jacobi}. 
					
					If $\ell=0$ then this equation is satisfied and $\G$ is a central extension of an abelian Lie algebra. 
					
					Suppose that $\ell\not=0$ and we can suppose $a_1\not=0$. Then for $x,y\in\ker\ell$ 
					\[ \ell(D(y))x-\ell(D(x))y=0 \]and hence $\ker\ell$ is invariant by $D$. If we take $x\in\ker\ell$ and $y\notin\ker\ell$ then
					\[ \ell(y)D(x)+\ell(D(y))x=0 \]and hence
					\[ D(x)=-\frac{\ell(D(y))}{\ell(y)}x. \]
					If we choose $y_0$ such that $\ell(y_0)\not=0$, we get
					\[ D(y_0)=\la y_0+x_0\esp D(x)=-\la x,x,x_0\in\ker\ell. \]Put $g_1=f_1$ and for $i=2,\ldots,k$, we put $g_i=a_1f_i-a_if_1\in\ker\ell$. Thus
					\[ D(g_1)=\la g_1+\sum_{i=2}^kb_ig_i,\;\theta^1=g_1^*\wedge e^*-\sum_{i=2}^ka_ig_i^*\wedge e^*\esp D(g_i)=-\la g_i,\; \theta^i=a_1g_i^*\wedge e^* ,\quad i=2,\ldots,k.\]
									\[  \]and the Lie brackets are
					\[ [e,g_1]=a_1e+\la g_1+\sum_{l=2}^kb_lg_l, [e,g_i]=-\la g_i,[g_1,g_i]=a_1g_i,[g_i,g_j]=0, i,j=2,\ldots,k. \]This completes the proof.
					\end{proof}

			\section{Six dimensional  $2$-para-K\"ahler Lie algebras}\label{section5}
	
	In this section, by using Theorem \ref{main}, we give all six dimensional $2$-para-K\"ahler Lie algebras. We proceed as follows:
	\begin{enumerate}\item In Table \ref{1}, we determine all  $2$-left symmetric algebras by a direct computation using
		 the classification of real two dimensional left symmetric algebras given in \cite{K}. 
		 \item In Table \ref{2}, we give for each 2-left symmetric algebra in Table \ref{1} its compatible $2\times2$-left symmetric algebras.
		 \item In Table \ref{3}, we give for each couple of compatible structures in Table \ref{2} the corresponding $2$-para-K\"ahler Lie algebra.
		 \item All our computations were checked by using the software Maple.

		\end{enumerate}

	{\renewcommand*{\arraystretch}{1.6}
	\begin{tabular}{|c|l|l|}
		\hline
	Name of the	2-LSS &First left symmetric product& Second left symmetric product\\
		\hline
	$\mathbf{b}_{1,\al},(\al\not=1,\al\not=\frac12)$&	$e_2\bullet_1e_1=e_1,e_2\bullet_1e_2=\al e_2$&$\bullet_2=a\bullet_1$\\
		\hline
	$\mathbf{b}_{1,\frac12}$&	$e_2\bullet_1e_1=e_1,e_2\bullet_1e_2=\frac12 e_2$&$e_2\bullet_2e_1=ae_1,e_2\bullet_2e_2=\frac12a e_2+be_1$\\
	\hline
	\multirow{2}{*}{$\mathbf{b}_{1,1}$}&	\multirow{2}{*}{$e_2\bullet_1e_1=e_1,e_2\bullet_1e_2= e_2$}&$e_1\bullet e_1=ae_1,e_1\bullet_2e_2=ae_2,e_2\bullet_2e_1=be_1,$\\&&$e_2\bullet_2e_2=b e_2$\\	
	\hline
	$\mathbf{b}_{2}$&	$e_2\bullet_1e_1=e_1,e_2\bullet_1e_2= e_1+e_2$&$\bullet_2=a\bullet_1$\\	
	\hline
	$\mathbf{b}_{3,\al},\al\not=1,\al\not=0,$&	$e_1\bullet_1 e_2=e_1,e_2\bullet_1e_1=(1-\frac1\al)e_1,e_2\bullet_1e_2= e_2$&$\bullet_2=a\bullet_1$\\	
	\hline
\multirow{2}{*}{	$\mathbf{b}_{3,1}$}&\multirow{2}{*}	{$e_1\bullet_1 e_2=e_1,e_2\bullet_1e_2= e_2$}&$e_1\bullet_2e_1=ae_1,e_1\bullet_2e_2=be_1,e_2\bullet_2e_1=ae_2,$\\&&$e_2\bullet_2e_2=be_2$\\	
	\hline
$\mathbf{b}_{4}$&	$e_1\bullet_1e_2=e_1,e_2\bullet_1e_2= e_1+e_2$&$\bullet_2=a\bullet_1$\\	
\hline
$\mathbf{b}_{5}^+$&	$e_1\bullet_1e_1=e_2,e_2\bullet_1e_1= -e_1,e_2\bullet_1e_2=-2e_2$&$\bullet_2=a\bullet_1$\\	
\hline
$\mathbf{b}_{5}^-$&	$e_1\bullet_1e_1=-e_2,e_2\bullet_1e_1= -e_1,e_2\bullet_1e_2=-2e_2$&$\bullet_2=a\bullet_1$\\	
\hline
$\mathbf{c}_2$&$e_2\bullet_1 e_2=e_2$&$e_1\bullet_2e_1=ae_1,e_2\bullet_2e_2=be_2$\\
\hline
$\mathbf{c}_3^1$&$e_2\bullet_1 e_2=e_1$&$e_2\bullet_2e_1=2ae_1,e_2\bullet_2e_2=be_1+ae_2$\\
\hline
$\mathbf{c}_3^2$&$e_2\bullet e_2=e_1$&$e_1\bullet_2e_2=ae_1,e_2\bullet_2e_1=ae_1,e_2\bullet_2e_2=be_1+ae_2$\\
\hline
$\mathbf{c}_4$&$e_2\bullet e_2=e_2,e_1\bullet_1e_2=e_2\bullet_1e_1=e_1$&$e_1\bullet_2e_2=ae_1,e_2\bullet_2e_1=ae_1,e_2\bullet_2e_2=be_1+ae_2$\\
\hline
\multirow{2}{*}{$\mathbf{c}_5^+$}&\multirow{2}{*}{$e_1\bullet_1e_1=e_2\bullet_1 e_2=e_2,e_1\bullet_1e_2=e_2\bullet_1e_1=e_1$}&$e_1\bullet_2e_2=e_2\bullet_2e_1=be_1+ae_2$\\&&$e_1\bullet_2e_1=e_2\bullet_2e_2=ae_1+be_2$\\
\hline
\multirow{2}{*}{$\mathbf{c}_5^-$}&\multirow{2}{*}{$e_1\bullet_1e_1=-e_2\bullet_1 e_2=-e_2,e_1\bullet_1e_2=e_2\bullet_1e_1=e_1$}&$e_1\bullet_2e_2=e_2\bullet_2e_1=be_1+ae_2$\\&&$e_1\bullet_2e_1=-e_2\bullet_2e_2=ae_1-be_2$\\
\hline			
		\end{tabular}\captionof{table}{Two dimensional 2-left symmetric structures, $(a,b)\in\R^2$ .\label{1}}}
	{\renewcommand*{\arraystretch}{1.6}
		\begin{tabular}{|c|c|c|c|}
			\hline
			Name	&2-left symmetric structure &Compatible $(2\times 2)$-left symmetric structure &conditions\\
			\hline
			
			$\mathbf{bb}_{1,\al}$&	$\mathbf{b}_{1,\al},(\al\not=1,\al\not=\frac12)$& $L^{1,1}_{e_{2}}= \left ( \begin{smallmatrix} 0 & 0 \\ 0 & -a c \end{smallmatrix}\right) $, $L^{1,2}_{e_{2}}=\left ( \begin{smallmatrix} 0 & 0 \\ 0 & -ad \end{smallmatrix}\right)$, $L^{2,1}_{e_{2}}=\left ( \begin{smallmatrix} 0 & 0 \\ 0 & c \end{smallmatrix}\right)$, $L^{2,2}_{e_{2}}=\left ( \begin{smallmatrix} 0 & 0 \\ 0 & d \end{smallmatrix}\right)$ & $a\in\R,\ \al=0$ \\
			\hline

			$\mathbf{bb}_{1,1}$&	$\mathbf{b}_{1,1}$&$\star_{\al,\be}=0,\al,\be\in\{1,2\}$ & $a=0,\ b\in\R$\\
			
			\hline

			$\mathbf{bb}_{2}$ &$\mathbf{b}_{2}$ &$\star_{\al,\be}=0,\al,\be\in\{1,2\}$ & $ a\not=1$\\
			\cline{3-4}
			& & $L^{1,1}_{e_{2}}=L^{1,2}_{e_{2}}= \left ( \begin{smallmatrix} 0 & 0 \\ 0 & -c \end{smallmatrix}\right) $, $L^{2,1}_{e_{2}}=L^{2,2}_{e_{2}}= \left ( \begin{smallmatrix} 0 & 0 \\ 0 &  c \end{smallmatrix}\right)$ &$a=1$\\
			\hline
			
			\multirow{1}{*}{	$\mathbf{bb}_{3,1}$} &	\multirow{1}{*}{	$\mathbf{b}_{3,1}$} & $\star_{\al,\be}=0,\al,\be\in\{1,2\}$ & $a\not=0,\ b\in\R$\\
			
			\hline
			\multirow{1}{*}{	$\mathbf{bb}_{4}$} &	\multirow{1}{*}{	$\mathbf{b}_{4}$} & $L^{1,1}_{e_{1}}= \left ( \begin{smallmatrix} 0 & 0 \\ -ac & 0 \end{smallmatrix}\right),\ L^{1,2}_{e_{1}}= \left ( \begin{smallmatrix} 0 & 0 \\ -a^{2}c & 0 \end{smallmatrix}\right),\ L^{2,1}_{e_{1}}= \left ( \begin{smallmatrix} 0 & 0 \\ c & 0 \end{smallmatrix}\right),\ L^{2,2}_{e_{1}}= \left ( \begin{smallmatrix} 0 & 0 \\ ac & 0 \end{smallmatrix}\right) $ & $ a\in\R$\\
			\hline
			
			$\mathbf{cc}_3^1$ &$\mathbf{c}_3^1$ &$\star_{\al,\be}=0,\al,\be\in\{1,2\}$  & $a\not=0,\ b\in\R$  \\
			\cline{3-4}
			& & $L^{1,1}_{e_{1}}= \left ( \begin{smallmatrix} c_1 & 0 \\ c_2 & 0 \end{smallmatrix}\right),\ L^{1,2}_{e_{1}}= \left ( \begin{smallmatrix} b c_1 & 0 \\  d_1 & 0 \end{smallmatrix}\right),\ L^{2,1}_{e_{1}}= \left ( \begin{smallmatrix} g_1 & 0 \\ g_2 & 0 \end{smallmatrix}\right),\ L^{2,2}_{e_{1}}= \left ( \begin{smallmatrix} b g_1 & 0 \\ d_2 & 0 \end{smallmatrix}\right) $ & $a=0,\ b\in\R$\\
			\hline
			$\mathbf{cc}_3^2$ &$\mathbf{c}_3^2$ &$\star_{\al,\be}=0,\al,\be\in\{1,2\}$  & $a\not=0,\ b\in\R$  \\
			\cline{3-4}
			& & $L^{1,1}_{e_{1}}= \left ( \begin{smallmatrix} c_1 & 0 \\ c_2 & 0 \end{smallmatrix}\right),\ L^{1,2}_{e_{1}}= \left ( \begin{smallmatrix} b c_1 & 0 \\  d & 0 \end{smallmatrix}\right),\ L^{2,1}_{e_{1}}= \left ( \begin{smallmatrix} g_1 & 0 \\ g_2 & 0 \end{smallmatrix}\right),\ L^{2,2}_{e_{1}}= \left ( \begin{smallmatrix} b g_1 & 0 \\ h & 0 \end{smallmatrix}\right) $ & $a=0,\ b\in\R$\\
			\hline
			
			\multirow{1}{*}{$\mathbf{cc}_5^+$} &\multirow{1}{*}{$\mathbf{c}_5^+$} &$L^{1,1}_{e_{1}}=L^{1,1}_{e_{2}}= \left ( \begin{smallmatrix} -c & -c \\ -c & -c \end{smallmatrix}\right),\ L^{1,2}_{e_{1}}=L^{1,2}_{e_{2}}= \left ( \begin{smallmatrix} -c & -c \\ -c & -c \end{smallmatrix}\right) $	& $a\in\R,\ b\in\R$  \\		
			& & $L^{2,1}_{e_{1}}=L^{2,1}_{e_{2}}= \left ( \begin{smallmatrix} c & c \\ c & c \end{smallmatrix}\right),\ L^{2,2}_{e_{1}}=L^{2,2}_{e_{2}}= \left ( \begin{smallmatrix} c & c \\ c & c \end{smallmatrix}\right) $	&   \\		
			\hline
			
		\end{tabular}}\captionof{table}{Compatible two dimensional $2$-left symmetric and $(2\times 2)$-left symmetric structures.\label{2}}

		{\renewcommand*{\arraystretch}{1.6}
			\begin{tabular}{|c|c|c|}
				\hline
				 Structure	&Associated $2$-para-K\"ahler Lie algebra  &Conditions \\
				\hline
				\multirow{2}{*}{$\mathbf{bb}_{1,\al}$}&	$[f_{1},f_{2}]=-f_{1},\ [f_{1},f_{4}]=-a f_{1}, \ [f_{2},f_{3}]=f_{3}, \ [f_{3},f_{4}]=-a f_{3}, $  &$a\in\R,\ $\\
				&$ [f_{2},e_{1}]=-e_{1},\ [f_{2},e_{2}]=-c(af_2-f_4), \ [f_{4},e_{1}]= -a e_{1}, \ [f_{4},e_{2}]=-d(af_2-f_4). $&\\
				
				\hline
				
				\multirow{2}{*}{$\mathbf{bb}_{1,1}$} &$[f_{1},f_{2}]=-f_{1},\ [f_{1},f_{4}]=-b f_{1},\ [f_{2},f_{3}]=f_{3},\ [f_{2},f_{4}]=-b f_{2}+f_{4},\ [f_{3},f_{4}]=-b f_{3},   $ &$\ b\in\R$\\	
				& $[f_{2},e_{1}]=-e_{1},\ [f_{2},e_{2}]=-e_{2},\ [f_{4},e_{1}]=-b e_{1},\ [f_{4},e_{2}]=-b e_{2}. $&\\
				
				\hline
				\multirow{5}{*}{$\mathbf{bb}_{2}$} & $[f_{1},f_{2}]=-f_{1},\ [f_{1},f_{4}]= -a f_{1},\ [f_{2},f_{3}]= f_{3},\ [f_{2},f_{4}]=-a( f_{1}+ f_{2})+f_{3}+f_{4},  $ &$ a\not=1$\\			
				& $[f_{3},f_{4}]=-a f_{3},\ [f_{2},e_{1}]= -e_{1}-e_{2},\ [f_{2},e_{2}]= -e_{2},\ [f_{4},e_{1}]= -a( e_{1}+ e_{2}), [f_{4},e_{2}]= -a e_{2}.  $ &\\
				\cline{2-3}
				& $ [f_{1},f_{2}]=-f_{1},\  [f_{1},f_{4}]= - f_{1},\ [f_{2},f_{3}]= f_{3},\ [f_{2},f_{4}]=- f_{1}- f_{2}+f_{3}+f_{4},  $ &$c\in\R$ \\	
				& $ [f_{3},f_{4}]=- f_{3},\  [f_{2},e_{1}]= -e_{1}-e_{2},\ [f_{2},e_{2}]=-c(f_{2}-f_{4}) -e_{2},\ [f_{4},e_{1}]= -e_{1}-e_{2} ,  $&\\
				& $[f_{4},e_{2}]=-c(f_{2}-f_{4})-e_{2}.$&\\	
				\hline

				\multirow{2}{*}{	$\mathbf{bb}_{3,1}$} &$ [f_{1},f_{2}]=f_{1},\ [f_{1},f_{3}]= -a f_{1},\ [f_{1},f_{4}]=-a f_{2}+f_{3},\ [f_{2},f_{3}]= -b f_{1}, $ &$a\not=0,\ b\in\R$\\	
				& $ [f_{2},f_{4}]=-b f_{2}+f_{4},\ [f_{3},f_{4}]=b f_{3}-a f_{4},\ [f_{1},e_{1}]=- e_{2},\ [f_{2},e_{2}]=- e_{2}, $ &\\
				& $ [f_{3},e_{1}]=-ae_{1}-be_{2},\ [f_{4},e_{2}]=-ae_{1}-be_{2}.$ &\\
				
				\hline

					\multirow{1}{*}{	$\mathbf{bb}_{4}$} &$ [f_{1},f_{2}]=f_{1},\ [f_{1},f_{4}]=f_{3},\ [f_{2},f_{3}]= -a f_{1},\  [f_{2},f_{4}]=-a (f_{1}+ f_{2})+f_{3}+f_{4},  $ &$ a\in\R$ \\
					& $ [f_{3},f_{4}]=a f_{3},\ [f_{1},e_{1}]= -e_{2},\ [f_{2},e_{1}]= -c(a f_{1}-f_{3})-e_{2},\ [f_{2},e_{2}]= -e_{2}, $&\\
					& $[f_{3},e_{1}]= -a e_{2}, [f_{4},e_{1}]= -a c(af_{1}- f_{3})-a e_{2},\ [f_{4},e_{2}]= -a e_{2}.   $ &\\
					\hline
					
					\multirow{1}{*}{	$\mathbf{cc}_3^1$} & $ [f_{1},f_{4}]=-2a f_{1},\   [f_{2},f_{4}]=-b f_{1}-a f_2+ f_{3},\ [f_{3},f_{4}]=-2a f_{3},\ [f_{2},e_{1}]=-e_2,  $  &	$a\not=0,\ b\in\R$\\
					& $[f_{4},e_{1}]= -2a e_{1}-b e_2,\ [f_{4},e_{2}]= -a e_{2}.$ &\\
					\cline{2-3}
					& $   [f_{2},f_{4}]=-b f_{1}+ f_{3},\ [f_{1},e_{1}]=c_1 f_1+g_1 f_3,\ [f_{2},e_{1}]=c_2 f_1+g_2 f_3-e_2,$ & $\ b\in\R$\\
					&$[f_{3},e_{1}]= b c_1 f_1 + b g_1 f_3,\ [f_{4},e_{1}]= d_2 f_1+h f_3-b e_2.  $  &	\\
					\hline
					\multirow{1}{*}{	$\mathbf{cc}_3^2$} & $ [f_{1},f_{4}]=-a f_{1},\ [f_{2},f_{3}]=-a f_{1},\   [f_{2},f_{4}]=-b f_{1}-a f_2+ f_{3},\  [f_{2},e_{1}]=-e_2,\ [f_{3},e_{1}]=-ae_2  $  &	$a\not=0,\ b\in\R$\\
					& $[f_{4},e_{1}]= -a e_{1}-b e_2,\ [f_{4},e_{2}]= -a e_{2}.$ &\\
					\cline{2-3}
					& $   [f_{2},f_{4}]=-b f_{1}+ f_{3},\ [f_{1},e_{1}]=c_1 f_1+g_1 f_3,\ [f_{2},e_{1}]=c_2 f_1+g_2 f_3-e_2,$ & $\ b\in\R$\\
					&$[f_{3},e_{1}]= b c_1 f_1 + b g_1 f_3,\ [f_{4},e_{1}]= d f_1+h f_3-b e_2.  $  &	\\
					\hline
					
					\multirow{1}{*}{$\mathbf{cc}_5^+$} & $ [f_{1},f_{3}]=-a f_{1}-bf_2+f_4,\  [f_{1},f_{4}]=-b f_{1}-af_2+f_3,\ [f_{2},f_{3}]=-b f_{1}-af_2+f_3,$ & $a\in\R,\ b\in\R$\\		
					&$[f_{2},f_{4}]=-a f_{1}-bf_2+f_4,\ [f_{1},e_{1}]=-c(f_1+f_2-f_3-f_4)-e_2,\  [f_{1},e_{2}]=-c(f_1+f_2-f_3-f_4)-e_1, $ & \\
					& $ [f_{2},e_{1}] =-c(f_1+f_2-f_3-f_4)-e_1,\ [f_{2},e_{2}] =-c(f_1+f_2-f_3-f_4)-e_2  $ &\\ 
					& $ [f_{3},e_{1}] =-c(f_1+f_2-f_3-f_4)-a e_1-be_2,\ [f_{3},e_{2}] =-c(f_1+f_2-f_3-f_4)-b e_1-ae_2 , $ &\\
					& $ [f_{4},e_{1}] =-c(f_1+f_2-f_3-f_4)-b e_1-ae_2 ,\ [f_{4},e_{2}] =-c(f_1+f_2-f_3-f_4)-a e_1-be_2 .   $& \\
					\hline
					&$\h=\mathrm{span}\{f_1,f_2,f_3,f_4  \}$, $\;\theta^1=f_1^*\wedge e_1^*+f_2^*\wedge e_2^*\esp \theta^2=f_3^*\wedge e_1^*+f_4^*\wedge e_2^*$&\\
					\hline

				\end{tabular}} \captionof{table}{Six dimensional $2$-para-K\"ahler Lie algebras.\label{3}}

\end{document}